\documentclass[12pt]{amsart}
\usepackage[margin=3cm]{geometry}

\usepackage{graphicx} 
\usepackage{amsmath}
\numberwithin{equation}{section}
\usepackage{amssymb}
\usepackage{amsfonts}
\usepackage{mathtools}
\usepackage{tikz}
\usepackage{pgfplots}
\usepackage{subcaption}

\let\mc=\mathcal
\let\mr=\mathrm
\let\bs=\boldsymbol

\def\ed{\mathrm{d}}
\usepackage{xcolor}
\usepackage[
  colorlinks=true,
  linkcolor=blue,
  citecolor=blue,
  urlcolor=blue
]{hyperref}
\usepackage{amsthm}
\theoremstyle{plain}
\newtheorem{theorem}{Theorem}[section]
\newtheorem{remark}[theorem]{Remark}
\newtheorem{lemma}[theorem]{Lemma}

\newtheorem{corollary}[theorem]{Corollary}
\newtheorem{proposition}[theorem]{Proposition}

\theoremstyle{definition}

\newtheorem{example}[theorem]{Example}

\usepackage{mdframed}[framemethod=TikZ]

\usepackage{array,multirow,makecell}
\setcellgapes{1pt}
\makegapedcells
\newcolumntype{R}[1]{>{\raggedleft\arraybackslash }b{#1}}
\newcolumntype{L}[1]{>{\raggedright\arraybackslash }b{#1}}
\newcolumntype{C}[1]{>{\centering\arraybackslash }b{#1}}

\graphicspath{{}}

\title[Duality and particle approximations for signed Wasserstein barycenters]{Toland duality and particle approximations for signed Wasserstein barycenters}

\author[T.\ O. Gallou\"et]{Thomas O.\ Gallou\"et}
\address{Thomas O. Gallou\"et, Université Paris-Saclay, CNRS, Inria, Laboratoire de mathématiques d’Orsay, ParMA, 91405, Orsay, France}
\email{thomas.gallouet@inria.fr}

\author[A. Natale]{Andrea Natale}
\address{Andrea Natale, Université Paris-Saclay, CNRS, Inria, Laboratoire de mathématiques d’Orsay, ParMA, 91405, Orsay, France}
\email{andrea.natale@inria.fr}

\author[G. Todeschi]{Gabriele Todeschi}
\address{Gabriele Todeschi, Université Claude Bernard Lyon 1, ICJ UMR5208, CNRS, Ecole Centrale de Lyon, INSA Lyon, Université
Jean Monnet, 69622 Villeurbanne, France}
\email{todeschi@math.univ-lyon1.fr}
\begin{document}

\maketitle

\begin{abstract}
We study the problem of minimizing a weighted sum of squared Wasserstein distances with signed coefficients. In the case of a single positive coefficient, we derive two convex dual formulations: one expressed in terms of Brenier potentials, and one as a projection problem in convex order. Such dual formulations arise via Toland duality by exploiting different notions of convexity within the problem. They generalize those recently established for the metric extrapolation problem, which corresponds to signed barycenters between only two measures. For general signed barycenters, we identify necessary optimality conditions that reduce the problem to the case with one positive coefficient. On the numerical side, we propose a grid-free, particle-based algorithm built on entropic regularization and Sinkhorn iterations to compute signed barycenters with arbitrary coefficients.
\end{abstract}

\section{Introduction}
Computing interpolations and extrapolations of measure-valued data is a common task in statistics and numerical analysis. One of the most successful frameworks in this context is that of Wasserstein barycenters, introduced by Agueh and Carlier in \cite{agueh2011barycenters}.  The Wasserstein barycenter of a family of probability measures 
 with finite second moments $\nu_1,\ldots, \nu_N\in\mc{P}_2(\mathbb{R}^d)$ is defined as any minimizer of
\begin{equation}\label{eq:signed}
\inf_\mu \sum_{i=1}^N \lambda_i \frac{W^2_2(\mu,\nu_i)}2 \,,
\end{equation}
where $\lambda_i>0$ are positive weights summing to one. The popularity of this notion mostly lies in its ability to generate interpolations that respect the geometry of the underlying space, while retaining a convex variational structure amenable to computation.

One fundamental limitation of this classical definition is that since the weights are assumed to be positive one cannot use it to define extrapolations. To overcome this restriction, one is naturally led to consider
signed Wasserstein barycenters, i.e., solutions of the same variational problem with weights $\lambda_i\in \mathbb{R}\setminus\{0\}$ 
satisfying only $\sum_i \lambda_i >0$. These arise naturally in many contexts including regression models \cite{petersen2019frechet,dubey2020functional,schotz2021frechet}, reduced order modeling \cite{dalery2023nonlinear}, or the design of time discretizations for gradient flows \cite{Matthes2019bdf2,gallouet2024geodesic,han2023high}. In the latter application, for example, the variational structure of the problem is particularly beneficial as it allows one to derive convergence results that are not accessible via linearization-based approaches \cite{gallouet2024geodesic}.
However, the non-convexity introduced by negative weights means that standard existence, uniqueness, and duality arguments for classical barycenters break down entirely. Consequently, numerical solutions can be reliably computed only in dimension one or, due to recent developments, in some other specific settings \cite{nguyen2026fr,gallouet2025metric}.

The aim of this work is to address these gaps by providing a theoretical framework to characterize the solutions of problem \eqref{eq:signed} with signed weights, and practical numerical tools to approximate them.  

\subsection{Metric extrapolation and convex order} Our starting point is the notion of metric extrapolation introduced in \cite{gallouet2025metric}. Given two measures with finite second moments $\nu_0,\nu_1 \in \mc{P}_2(\mathbb{R}^d)$, this is the unique solution to the following problem:
\begin{equation}\label{eq:extrapolation}
\inf_\mu \left \{ t \frac{W^2_2(\mu,\nu_1)}{2} - (t-1)\frac{W^2_2(\mu,\nu_0)}{2}  \right\}.
\end{equation}
To explain why this is an extrapolation, recall that a (constant-speed, minimizing) Wasserstein geodesic is a curve $\nu :s\in[0,t] \rightarrow \nu(s)$ satisfying
\begin{equation}\label{eq:geodesicd}
W_2(\nu(s_0),\nu(s_1)) = \frac{|s_1-s_0|}{t} W_2(\nu(0),\nu(t))\,, 
 \end{equation}   
 for all $s_0,s_1\in [0,t]$. 
 Using elementary inequalities, it is easy to check that if there exists a globally minimizing geodesic such that $\nu(0)=\nu_0$ and $\nu(1)=\nu_1$, then $\nu(t)$  is  a minimizer of \eqref{eq:extrapolation}; see \cite{gallouet2025metric}. When no such geodesic exist (which is often the case for arbitrary data $\nu_0,\nu_1$) problem \eqref{eq:extrapolation} still has a unique solution, which therefore provides a well-defined notion of geodesic extrapolation.

Remarkably, while problem \eqref{eq:extrapolation} is non convex with respect to the linear structure on probability measures, it admits a convex dual formulation, which can be written as follows:
\begin{equation}\label{eq:dualBrenier} 
\inf \left \{ \int u^* \mathrm{d} \nu_1  + \int u \mathrm{d} \nu_0 ~;~ u - \frac{t-1}{t} \frac{|\cdot|^2}{2} \text{ is proper, convex {and l.s.c.}} \right\},
\end{equation}
where $u^*$ denotes the Fenchel-Legendre transform of $u$.
In particular, given a solution $u$ to this problem, define $\overline{\nu}_0 \coloneqq \nabla u^*_\# \nu_1$ (where the subscript $\#$ denotes the push-forward of measures). Then, $\overline{\nu}_0$ is dominated in convex order by $\nu_0$, denoted $\overline{\nu}_0 \preceq_C \nu_0$, i.e.,
\[
\int \varphi \ed \overline{\nu}_0 \leq \int \varphi \ed \nu_0 \,, \quad \forall \, \varphi:\mathbb{R}^d \rightarrow \mathbb{R} ~~~ \text{convex}\,.
\]
Moreover there exists a unique $W_2$-geodesic $\overline{\nu}:[0,t] \rightarrow \mc{P}_2(\mathbb{R}^d)$ such that $\overline{\nu}(0) = \overline{\nu}_0$,  $\overline{\nu}(1) = \nu_1$, and  $\overline{\nu}(t)$ solves \eqref{eq:extrapolation}. Finally, the measure $\overline{\nu}_0$ can be characterized explicitly as a $W_2$-projection in convex order \cite{kim2026stability}, i.e., it is the unique solution of
\begin{equation}\label{eq:dualprojection}
\inf \left\{ W^2_2 (D^\theta_\# \nu_1, \overline{\nu}_0) ~;~  \overline{\nu}_0 \preceq_C \nu_0 \right\}\,,
\end{equation}
where $D^\theta(x) \coloneqq  \theta x$ and $\theta = t/(t-1)$.

\subsection{Related works and contributions} In this work we 
generalize the results of \cite{gallouet2025metric} to signed barycenters with more than two measures. Specifically, we show that if one single coefficient $\lambda_i$ is positive and the rest are negative, problem \eqref{eq:signed} admits a convex dual formulation in terms of Brenier potentials analogous to \eqref{eq:dualBrenier} and a convex dual formulation as projection in convex order analogous to \eqref{eq:dualprojection}. Moreover, the latter can also be expressed as a weak Optimal Transport problem \cite{gozlan2017kantorovich,gozlan2020mixture} (a  generalization of Optimal Transport in which the cost depends nonlinearly on the plan $\gamma$) involving multiple plans, which makes the problem tractable numerically.

While the general case with multiple positive coefficients lacks the same underlying convex structure, we identify necessary conditions that characterize the solutions in terms of barycenters with one positive coefficient. These allow us to construct a counterexample to both uniqueness and convexity of the minimizers of the problem which involves measures that are all absolutely continuous. This addresses a question from \cite{jacobs2026signed}, where the authors proved uniqueness for absolutely continuous data, upon assuming that the solution of the problem can be constructed from a saddle point of a certain dual formulation, which therefore does not seem to hold in general. 

Note that other recent works have addressed the signed barycenter problem \cite{tornabene2024generalized,jacobs2026signed}. However, they do not develop the same duality theory as we do here, and do not identify the link between signed barycenters and convex order. Furthermore, the proofs provided in this work to establish the dual formulations of \eqref{eq:barycenter} are more direct than those we provided in \cite{gallouet2025metric} for the case $N=2$, and clarify the structure of the problem for this case as well. 

Here, our approach relies on reformulating the problem in terms of random variables (a classical approach in Optimal Transport that is sometimes referred to as Lions's lift), adopting the perspective already used in \cite{ekeland2014optimal,bertucci2024approximation,beiglbock2025brenier,pinzi2025totally} for related problems. The main observation we will use is that minimizing a difference of 2-Wasserstein distances squared is a equivalent to computing a Legendre transform after the lift; see Lemma \ref{lem:Fstar} below. This fact is a particular instance of a more general result due to Beiglb\"ock, Pammer and Schrott \cite{beiglbock2025brenier} stating that $c$-convexity with respect to the maximum covariance function (defined in equation \eqref{eq:mcov}) is effectively equivalent to standard convexity after the lift.  The equivalence of the metric extrapolation problem with a convex problem after the lift was also noticed in \cite{bertucci2024approximation}, and \cite[Remark 5.3]{bourne2026semidiscrete} in the semi-discrete setting. 

We use this point of view to derive dual formulations of the signed barycenter problem, using 
Toland duality \cite{toland1978duality}, a general principle for differences of convex functions, as well as classical minimax arguments. In fact, we show that the different dual formulations we derive correspond to different way of applying Toland duality, using either the linear structure on probability measures or the linear structure on random variables obtained after the lift; see Section \ref{sec:toland}. We note that there is a strong analogy between our dual formulations and those corresponding to the so-called Bass problem  \cite{backhoff2024gradient,BackhoffVeraguas2025Bass}, which in one of its reformulations is also equivalent to the minimization of a difference of squared Wasserstein distances, but where the unknown $\mu$ is convolved with a standard Gaussian in the positive term.

From a numerical perspective, we generalize the numerical approach developed in \cite{gallouet2025metric} to the case of multiple negative coefficients. While this relies on entropic regularization, it is a grid-free method in the sense that the support of the solution is given by a set of particles whose positions are the main unknowns. Building on this scheme, we propose a numerical algorithm for  computing approximate solutions to general signed barycenters. Finally,  we show how Toland duality connects our framework to optimal quantization and the computation of particle solutions of barycenters with positive coefficients. 

\subsection{Structure of the article} The rest of the article is structured as follows. In Section \ref{sec:mcov} we review the problem of minimizing the difference of $W_2^2$ costs, highlighting its link with convex order. This is used in Section \ref{sec:kto1} to study the signed barycenter problem with one positive coefficient, and derive different convex dual formulations. In Section \ref{sec:linearized} we connect this problem to linearized OT barycenters. In Section \ref{sec:general}, we study signed barycenters with arbitrary coefficients, deriving necessary optimality conditions and a saddle point formulation of the problem in terms of convex order conditions.  We discuss the computation of signed barycenters in Section \ref{sec:numerical}.  

\section{Maximum covariance and convex order}
\label{sec:mcov}
In this section we discuss the minimization of a difference of squared 2-Wasserstein distances. Specifically, given $\nu_0,\nu_1\in\mc{P}_2(\mathbb{R}^d)$ we study the problem
\begin{equation}\label{eq:diff}
\inf_\mu \left \{ \frac{W^2_2(\mu,\nu_1)}{2} - \frac{W^2_2(\mu,\nu_0)}{2}  \right\}\,,
\end{equation}
where 
 \[
 W^2_2(\mu,\nu) \coloneqq \inf \left\{ \int |x-y|^2 \ed \gamma(x,y) ~ ; ~\gamma \in \Gamma(\mu,\nu) \right\}\,,
 \]
and $\Gamma(\mu,\nu)$ denotes the set of couplings with fixed marginals equal to $\mu$ and $\nu$.
This can be regarded as a degenerate signed Wasserstein barycenter since the sum of the coefficients is exactly zero, but it will also be an essential first step to understand problem \eqref{eq:signed}. Problem \eqref{eq:diff} was originally studied by \cite{carlier2008Toland} (although on the set of probability measures on a compact subset of $\mathbb{R}^d$) and it is intimately connected with the notion of convex order. The main objective of this section is to explain such a connection, collecting results from \cite{carlier2008Toland,ekeland2014optimal,bertucci2024approximation,beiglbock2025brenier}, and describe the main tools and notation used in the rest of this work.

Throughout the rest of the paper, we will always work with a fixed atomless, separable probability space $(\Omega, \mathcal{F}, P)$, ensuring that $ L^2(\Omega, \mc{F}, P; \mathbb{R}^d)$ is a separable Hilbert space and that any measure  $\mu \in \mc{P}_2(\mathbb{R}^d)$ can be represented as the law of $X \in L^2(\Omega, \mc{F}, P; \mathbb{R}^d)$. We will denote this either by $X\sim \mu$ or $X_\# P = \mu$, where the subscript $\#$  denotes the pushforward of measures.

Let us start by recalling the definition of the maximum covariance function:
\begin{equation}\label{eq:mcov}
\mr{MCov}(\mu,\nu) = \sup\left\{ \mathbb{E}[\langle X,Y \rangle] ~ ; ~ X \sim \mu \,, ~Y \sim \nu \right\}\,.
\end{equation}
This is linked to $W_2$ by the following relation
\[
\frac{W^2_2(\mu,\nu)}{2} = \frac{M_2(\mu)}{2} + \frac{M_2(\nu)}{2} - \mr{MCov}(\mu,\nu)\,,
\]
where $M_2(\mu) \coloneqq \int |x|^2 \ed \mu(x)$ denotes the second moment of $\mu$.
Clearly, then,  minimizing \eqref{eq:diff} is equivalent to maximizing a difference of maximum covariance functions:
\begin{equation}\label{eq:diffmcov}
\inf_\mu \{\mr{MCov}(\mu,\nu_0) - \mr{MCov}(\mu,\nu_1)\}\,.
\end{equation}

The next step is to observe that while the definition of $\mr{MCov}$ involves a maximization over two random variables, one can always fix one of the two,  e.g., $X \sim \mu$, and perform the maximization only over all $Y \sim  \nu$; see for instance Lemma 2.1 in \cite{bertucci2024approximation}.
(Note that this is generally true for any optimal transport problem with an upper semi-continuous cost function  $c:\mathbb{R}^d\times \mathbb{R}^d \rightarrow \mathbb{R}$). 
Because of this, we can set
\begin{equation}\label{eq:Fnu}
F_\nu(X) \coloneqq \mathrm{MCov}(X_\# P, \nu)= \sup\left\{ \mathbb{E}[\langle X,Y \rangle] ~ ; ~ Y \sim \nu \right\}\,.
\end{equation}
We observe that by construction $F_\nu$ is convex and weakly l.s.c.\ as the supremum of linear functions. It is moreover Lipschitz continuous for the strong $L^2$ topology since $\nu\in\mc{P}_2(\mathbb{R}^d)$. 

\begin{remark}
In the language of \cite{pinzi2025totally}, equation \eqref{eq:Fnu} states that the functional $\mu \mapsto \operatorname{MCov}(\mu,\nu)$ is totally convex.
\end{remark}

Lemma \ref{lem:Fstar} below says that \eqref{eq:diff} amounts to computing the Legendre transform of $F_\nu$. This statement is a direct consequence of Proposition 1.9 in \cite{beiglbock2025brenier} (see also Exemple 2.10). We give here a proof that relies on the so-called semi-dual formulation of the maximum covariance function (see also \cite{bourne2026semidiscrete} for a different argument using Strassen's Theorem \cite{strassen1965existence}): 
\begin{multline}\label{eq:mcovsd}
\mr{MCov}(\mu,\nu) = \inf\left\{ \int u \ed \mu + \int u^* \ed \nu ~ ;~ \right. \\ \left. u:\mathbb{R}^d\rightarrow \mathbb{R}\cup\{+\infty\}  \text{ proper, convex and l.s.c.} \right \}\,.
\end{multline}
Note that for any $\mu,\nu \in \mc{P}_2(\mathbb{R}^d)$, problem \eqref{eq:mcovsd} always admits a solution \cite{villani2009optimal}, which we refer to as Brenier potential from $\mu$ to $\nu$. 

\begin{lemma}\label{lem:Fstar} For all $Y \in L^2$,
\begin{equation}\label{eq:Fstar}
F_{\nu_0}^*(Y) \coloneqq \sup_X  \{ \langle X,Y\rangle_{L^2} - F_{\nu_0}(X)\} = \sup_\mu \left\{\mr{MCov}(\mu,\nu_1) -\mr{MCov}(\mu,\nu_0) \right\}\,,
\end{equation}
where $\nu_1 = Y_\# P$. Moreover,
\begin{equation}\label{eq:Fnu0star}
F^*_{\nu_0}(Y) = \left\{ \begin{array}{ll} 0 & \text{if } Y_\# P \preceq_C \nu_0\,,\\
+\infty & \text{otherwise.} 
\end{array}\right.
\end{equation}
\end{lemma}

\begin{proof} Equation \eqref{eq:Fstar} is a direct consequence of the considerations above. In fact,
\[
\begin{aligned}
\sup_\mu \left\{\mr{MCov}(\mu,\nu_1) -\mr{MCov}(\mu,\nu_0) \right\} &= \sup_\mu  \sup_{X \sim \mu} \left\{\langle X,Y\rangle_{L^2}  -\mr{MCov}(\mu,\nu_0) \right\} \\
&= \sup_X \left\{ \langle X,Y\rangle_{L^2}  -\mr{MCov}(X_\# P,\nu_0) \right\}\,.
\end{aligned}
\]
For the second point, replacing $F_{\nu_0}(X)$ with the semi-dual formulation we obtain
\begin{equation}\label{eq:supsupF}
F^*_{\nu_0}(Y) = \sup_X \sup_u \left\{ \langle X,Y\rangle_{L^2} - \int u^*(X) \ed P - \int u \ed \nu_0 \right\}
\end{equation}
with $u$ convex, proper and l.s.c. At this point we can exchange the two suprema (upon assuming $ u \in \mc{L}^1(\nu_0)$), and using the measurable selection theorem \cite[Theorem 6.9.13]{bogachev2007measure}
we can pass the maximation over $X$ inside the integral (see \cite[Lemma B.2]{aubin2026debiasing}), leading to
\[
F^*_{\nu_0}(Y) =  \sup_u \left\{ \int \sup_x\{ \langle x,Y(\omega)\rangle - u^*(x)\}  \ed P(\omega)- \int u \ed \nu_0 \right\}\,.
\]
Since $u^{**} = u$, this gives
\[
F^*_{\nu_0}(Y) = \sup_u \left\{ \int u(Y) \ed P - \int u \ed \nu_0 ~;~ u \text{ proper, convex and l.s.c., } u \in \mc{L}^1(\nu_0)\right\}\,.
\]
\end{proof}

Since $F_\nu$ is convex and continuous, $F^{**}_\nu = F_\nu$. In particular, from Lemma \ref{lem:Fstar} we immediately obtain the following alternative defition of $\mr{MCov}$ (and therefore also of $W_2^2$, up to adding back the second moments):

\begin{lemma}\label{lem:mcov} For all $X\in L^2$ and $\nu \in \mc{P}_2(\mathbb{R}^d)$,
\[
F_\nu(X)=\mr{MCov}(X_\#P,\nu) = \sup_Y\{\mathbb{E}[\langle X,Y\rangle]~,~Y_\#P \preceq_C \nu\}\,,
\]
or equivalenlty for all $\mu,\nu \in \mc{P}_2(\mathbb{R}^d)$,
\[
\mr{MCov}(\mu,\nu) = \sup\left\{ \int\langle x,y\rangle\ed \gamma(x,y) ~ ; ~\gamma \in \Gamma(\mu,\bar{\nu}) \,,~ \bar{\nu} \preceq_C \nu \right\}\,.
\]
\end{lemma}

\subsection{Connection with Toland duality}\label{sec:toland} Given a Hilbert space $H$ and two proper, convex and l.s.c.\ functions $f,g: H \rightarrow (-\infty,+\infty]$, consider the minimization problem 
\begin{equation}\label{eq:fg}
\inf_{x\in \mathrm{dom}(f)} \{f(x) - g(x)\}\,. 
\end{equation}
Since $g = g^{**}$, we have that this is equal to 
\begin{equation}\label{eq:infinf}
\inf_{x\in \mathrm{dom}(f)} \inf_{y\in \mr{dom}(g^*)} \{f(x) - \langle x, y\rangle + g^*(y)\} \,,
\end{equation}
and swapping the two infima we get
\begin{equation}\label{eq:gsfs}
 \inf_{y\in \mr{dom}(g^*)} \{ g^*(y) - f^*(y)\}\,.
\end{equation}
This dual formulation is due to Toland \cite{toland1978duality}, and we refer to it as Toland dual of problem \eqref{eq:fg}.  
Note that the argument above does not require $f$ to be convex l.s.c.; if $f$ is also convex l.s.c., however, \eqref{eq:fg} is also the Toland dual of \eqref{eq:gsfs}. 

The idea of using Toland duality for problem \eqref{eq:diff}, or equivalently \eqref{eq:diffmcov}, was already put forward in \cite{carlier2008Toland}, where $f$ and $g$ are replaced by the functions $\mu \mapsto -\mr{MCov}(\mu,\nu_i)$, $i=0,1$, which are convex with respect to the linear structure on $\mc{P}_2(\mathbb{R}^d)$. Adopting this point of view, the equivalent of \eqref{eq:infinf} can be obtained by replacing $\mr{MCov}(\mu,\nu_0)$ in problem \eqref{eq:diffmcov} with the semi-dual formulation \eqref{eq:mcovsd}. This corresponds precisely to the calculations in the proof of Lemma \ref{lem:Fstar}, and in particular equation \eqref{eq:supsupF}.

There is however a second way to use Toland duality for problem \eqref{eq:diffmcov}. In fact, using equation \eqref{eq:Fnu}, we can also write
\[
\inf_\mu \{\mr{MCov}(\mu,\nu_0) - \mr{MCov}(\mu,\nu_1) \} = \inf_X \{ F_{\nu_0}(X) - F_{\nu_1}(X)\} \,.
\]
which is still a difference of convex functions but on $L^2$.
Then, by Toland duality,
\[
\inf_\mu \{\mr{MCov}(\mu,\nu_0) - \mr{MCov}(\mu,\nu_1) \} = \inf_Y \{ F^*_{\nu_1}(Y)-F^*_{\nu_0}(Y)~;~ Y \in \mr{dom}(F^*_{\nu_1})\}\,.
\]
Combining this with equation \eqref{eq:Fnu0star} directly implies \eqref{eq:Fstar}. As shown in the following sections, while for problem \eqref{eq:diff} these two ways of applying Toland duality give the same result, in the case of signed barycenters with different coefficients they will yield different dual formulations.

\section{Signed barycenters with one positive coefficient} \label{sec:kto1}

In this section, we consider the signed barycenter problem with one positive coefficient. For given $\nu_1 \in \mc{P}_2(\mathbb{R}^d)$ and $\nu_0^1, \ldots, \nu_0^K \in \mc{P}_2(\mathbb{R}^d)$, we consider the following problem:
\begin{equation}\label{eq:Kto1} \tag{$\mc{P}$}
\inf_\mu \left\{ t \frac{W^2_2(\mu,\nu_1)}{2} - (t-1) \sum_{k=1}^K \lambda_0^k \frac{W^2_2(\mu,\nu_0^k)}{2} \right\},
\end{equation}
where $t>1$ and $\lambda^k_0>0$ are such that $\sum_k \lambda^k_0 = 1$. Any signed barycenter with one positive coefficient can be written in this way. We write the coefficients in this form in analogy with the extrapolation problem, with which this problem shares many features. For this reason we will also refer to \eqref{eq:Kto1} as $K$-to-1 extrapolation.

\subsection{Existence and uniqueness}
The first step to study \eqref{eq:Kto1} is to rewrite as a minimization over random variables rather than measures. This gives the following existence and uniqueness result:

\begin{lemma}\label{lem:Kto1existence} Problem \eqref{eq:Kto1} admits a unique solution $\mu \in \mc{P}_2(\mathbb{R}^d)$. Moreover, $\mu = X_\#P$ where $X$ is the unique solution of 
\begin{equation}\label{eq:Kto1X}
\inf_{X} \left\{ \frac{1}{2} \|X\|^2_{L^2} -t \langle X,Y\rangle_{L^2}  + (t-1) \sum_{k}\lambda_0^k F_{k}(X) \right\}\,,
\end{equation}
where $Y$ is any random variable such that $ Y \sim \nu_1$, and $F_k \coloneqq F_{\nu_0^k}$.
\end{lemma}
\begin{proof} Equation \eqref{eq:Kto1X} can be derived by the same arguments as in Lemma \ref{lem:Fstar}. Existence and uniqueness for \eqref{eq:Kto1X} follows from the direct method since the objective is strongly convex and weakly l.s.c., as discussed in Section \ref{sec:mcov}.
\end{proof}
\begin{remark} Note that  the solution $X$ of problem  \eqref{eq:Kto1X} may be different for any fixed $Y \sim \nu_1$. However the law of $X$ is always the unique solution of problem \eqref{eq:Kto1}. Moreover, the plan $(X,Y)_\# P$ is always optimal for the transport from $\mu$ to $\nu_1$.
\end{remark}

\subsection{Dual formulation with Brenier potentials} Just as in Lemma \ref{lem:Fstar} we can use the semi-dual formulation of $\mr{MCov}$ to derive a dual formulation for problem \eqref{eq:Kto1} in terms of Brenier potentials. We note that following \cite{carlier2008Toland,gallouet2025metric} this could also be derived as an instance of Toland duality by viewing \eqref{eq:Kto1} as a difference of convex functions with respect to the linear structure on $\mc{P}_2(\mathbb{R}^d)$ (see Section \ref{sec:toland}).

Replacing all the $F_k$ in \eqref{eq:Kto1X} with the corresponding semi-dual formulation, we get (up to an additive constant)
\begin{equation}\label{eq:xw}
\inf_X \inf_{w_k} \left\{ \frac{1}{2} \|X\|^2_{L^2} - t \langle X,Y\rangle_{L^2}  +  (t-1)\sum_k \lambda^k_0\left( \int w_k^*(X)\ed P + \sum_{k} \int w_k \ed \nu_0^k  \right)\right\}\,.
\end{equation}

In order to derive an equivalent of \eqref{eq:dualBrenier}, let us introduce
\begin{equation}\label{eq:udef}
u(z^1, \ldots, z^K) \coloneqq  \frac{t-1}{2t} \Big|\sum_k \lambda^k_0 {z^k} \Big|^2 + \sum_k \frac{\lambda^k_0}{t} w_k(z^k)\,.
\end{equation}
and the map $D^\lambda: \mathbb{R}^d \rightarrow \mathbb{R}^{dK}$, defined by
\begin{equation}\label{eq:Ddef}
D^\lambda(x) = \left(\lambda^1 x, \ldots, \lambda^K x \right) \,.
\end{equation}

By direct computation, one can check that
\begin{equation}\label{eq:ustarD}
(u^*\circ D^\lambda )(y) - \frac{t}{t-1}\frac{|y|^2}{2}= \frac{1}{t(t-1)} \inf_x \left\{ \frac{|x|^2}{2} - t \langle x,y \rangle + (t-1) \sum_k \lambda_0^k w_k^*(x) \right\}\,.
\end{equation}

The right-hand side is, up to a multiplicative factor, the same minimization problem we obtain after exchanging the infima in \eqref{eq:xw}.
Hence, applying the measurable selection theorem (see \cite[Lemma B.2]{aubin2026debiasing}), we obtain the following equivalent problem
\begin{equation}\label{eq:dualBreniermm} \tag{${\mc{P}}^*$}
\inf \left \{ \int u^* \circ D^\lambda \ed \nu_1  + \int u \, \ed \otimes_k \nu_0^k ~:~ w_k \text{ is proper, convex and l.s.c.} \right\},
\end{equation}
where $u$ and $D^\lambda$ are defined in \eqref{eq:udef} and \eqref{eq:Ddef}. 

\begin{remark} Observe that when $K=1$, problem \eqref{eq:dualBreniermm} coincides with the dual formulation \eqref{eq:dualBrenier} of the metric extrapolation problem.
\end{remark}

The precise strong duality result between problems \eqref{eq:Kto1} and \eqref{eq:dualBreniermm} is stated in the following theorem:

\begin{theorem}\label{th:duality}
The following holds:
\begin{equation}\label{eq:PPstar}
\frac{1}{t(t-1)} \eqref{eq:Kto1} =  \eqref{eq:dualBreniermm}  - C
\end{equation}
where
\[
C = \frac{t-1}{2t}   \int \Big|\sum_k  \lambda_0^k {z^k} \Big|^2 \ed \otimes_k\nu_0^k(z^k) +
\frac{1}{2t}\sum_k \lambda_0^k M_2(\nu_0^k) + \frac{1}{2}  M_2(\nu_1)\,. \]
Moreover, 
\begin{enumerate}
\item let $\mu$ solve \eqref{eq:Kto1}, and let $w_k$ be optimal Brenier potentials for the transport from $\nu^k_0$ to $\mu$ for all $k = 1, \ldots,K$. Then $(w_k)_k$ solve \eqref{eq:dualBreniermm} and
\begin{equation}\label{eq:optimpot}
\bar{u} \coloneqq t\frac{|\cdot|^2}{2} - (t-1) (u^*\circ D^\lambda)
\end{equation}
is an optimal Brenier potential for the transport from $\nu_1$ to $\mu$;
\item conversely, let $(w_k)_k$ solve \eqref{eq:dualBreniermm}. Then, the unique solution of problem \eqref{eq:Kto1} is given by 
\begin{equation}\label{eq:Tfromu}
\mu = (T + t(\mathrm{Id} - T))_\# \nu_1, \quad \text{where}\quad T \coloneqq \nabla (u^*\circ D^\lambda )\,,
\end{equation}
and $w_k$ are optimal Brenier potentials for the transport from $\nu^k_0$ to $\mu$.
\end{enumerate}
\end{theorem}

Before providing the proof, we make the following preliminary observation. Define
\begin{equation}\label{eq:defv}
v(x) \coloneqq \frac{t-1}{t} \frac{|x|^2}{2} + \inf_{z^k}\left\{ \sum_k \frac{\lambda_0^k}t w_k(z^k) ~;~ \sum_k \lambda_0^k z^k = x\right\}\,.
\end{equation}
By direct computation, one can verify that $u^* \circ D^\lambda = v^*$. Since $v$ is a strongly convex function, by classical convex analysis arguments we can deduce that $T= \nabla v^*$ is Lipschitz, and the potential in \eqref{eq:optimpot} is a smooth convex function. In particular, equation \eqref{eq:Tfromu} is well-defined even if $\nu_1$ is not absolutely continuous.

\begin{proof} Equation \eqref{eq:PPstar} is a direct consequence of the computations above. The same holds for the fact that the functions $w_k$ are optimal Brenier potentials from $\nu_0^k$ to $\mu$.
Now, since $\bar{u}$ is a smooth convex function, it is a competitor for the semi-dual problem for the transport from $\nu_1$ to $\mu$. We have
\begin{multline}\label{eq:Kto1est}
\frac{1}{t(t-1)}\eqref{eq:Kto1} - \frac{1}{2(t-1)} M_2(\nu_1) - \sum_k \frac{\lambda_0^k}{2t}M_2(\nu_0^k) \\\geq \frac{1}{2t(t-1)}  M_2(\mu) -\frac{1}{t-1}\left(\int \bar{u}^*\ed \mu + \int \bar{u} \ed \nu_1\right) + \sum_k \frac{\lambda_0^k}{t}\left( \int w_k^*\ed \mu + \int w_k \ed \nu_0^k \right)\,.
\end{multline}
Now, multiplying both sides of equation \eqref{eq:ustarD} by 
$-(t-1)$, we obtain
\begin{equation}\label{eq:ubarss}
\bar{u} = \left[ \frac{|\cdot|^2}{2t} + \frac{t-1}{t} \sum_k \lambda_0^k w^*_k  \right]^*.
\end{equation}
Taking the Legendre transform of both sides, we obtain that the right-hand side of \eqref{eq:Kto1est}  is bounded from below by
\[
\eqref{eq:dualBreniermm} - \frac{t}{2(t-1)}M_2(\nu_1) +\frac{t-1}{2t}   \int \Big|\sum_k  \lambda_0^k {z^k} \Big|^2 \ed \otimes_k\nu_0^k(z^k)\,.
\]
But then from the equality \eqref{eq:PPstar}, $\bar{u}$ is necessarily optimal.

The second point can be obtained in a similar way. First, we observe that $\bar{u}$ is an optimal Brenier potential for the transport from $\nu_1$ to $\mu$, as
\[
T + t(\mathrm{Id} -T) = \nabla \bar{u}
\]
is the gradient of the convex function $\bar{u}$. On the other hand, we can use $\mu$ as a competitor in problem \eqref{eq:Kto1}  and $w_k$ as a competitor for the semi-dual problems defining the transport from $\nu_0^k$ to $\mu$. This gives again \eqref{eq:Kto1est} but with $\geq$ replaced with $\leq$. Using again \eqref{eq:ubarss}, this time we obtain
\begin{multline}\label{eq:Kto1est2}
\frac{1}{t(t-1)}\eqref{eq:Kto1} - \frac{1}{2(t-1)} M_2(\nu_1) - \sum_k \frac{\lambda_0^k}{2t}M_2(\nu_0^k) \\\leq \frac{1}{2t(t-1)}  M_2(\mu) -\frac{1}{t-1}\left(\int \bar{u}^*\ed \mu + \int \bar{u} \ed \nu_1\right) + \sum_k \frac{\lambda_0^k}{t}\left( \int w_k^*\ed \mu + \int w_k \ed \nu_0^k \right) \\
= \eqref{eq:dualBreniermm} - \frac{t}{2(t-1)}M_2(\nu_1) +\frac{t-1}{2t}   \int \Big|\sum_k  \lambda_0^k {z^k} \Big|^2 \ed \otimes_k\nu_0^k(z^k)\,.
\end{multline}
But then from the equality \eqref{eq:PPstar}, $\mu$ is necessarily optimal for \eqref{eq:Kto1}  and $w_k$ are optimal Brenier potentials for the transport from $\nu^k_0$ to $\mu$.

\end{proof}

Since by Lemma \ref{lem:Kto1existence} \eqref{eq:Kto1} admits a solution, then as a consequence of the first point of Theorem \ref{th:duality}, and dual attainment for the optimal transport problem with the $L^2$ cost (see, e.g., \cite[Theorem 5.10]{villani2009optimal}), we immediately obtain the following:

\begin{corollary}[Dual attainement] Problem \eqref{eq:dualBreniermm} admits at least a solution $(w_k)_k$. Moreover, for all $k$, if $\nu_0^k$ is absolutely continuous, $\nabla w_k$ is unique $\nu_0^k$-a.e..
\end{corollary}

\begin{remark}[Uniqueness] \label{rem:uniqueness} In the setting of the second point of Theorem \ref{th:duality}, one can verify that the curve
\[
s\in[0,t] \mapsto (T +s(\mr{Id}-T))_\# \nu_1 \in\mc{P}_2(\mathbb{R}^d)
\]
is a minimizing geodesic with respect to $W_2$. This implies that, without further assumptions of absolute continuity, the maps $T \in L^2(\nu_1)$ and the measure $\overline{\nu}_0 \coloneqq T_\#\nu_1$ are always uniquely defined. In other words, problem \eqref{eq:dualBreniermm} yields a uniquely defined measure $\overline{\nu}_0$ such that the Wasserstein geodesic from $\overline{\nu}_0$ to $\nu_1$ stays minimizing up to time $t$. The arrival point at time $t$ is precisely the solution $\mu$ of problem \eqref{eq:Kto1}.
\end{remark}

\subsection{Dual formulation as projection in convex order}
We now deduce another dual formulation of \eqref{eq:Kto1}, which clarifies the structure of its solutions in terms of convex order relations. In particular, it will provide  an explicit variational characterization of the measure $\overline{\nu}_0$ in Remark \ref{rem:uniqueness}.

The starting point is again equation \eqref{eq:Kto1X}. Using Lemma \ref{lem:mcov}, we write this as follows:
\begin{equation}\label{eq:saddle}
\inf_{X} \sup_{Z^k} \left\{ \frac{1}{2} \|X\|^2_{L^2} -t \langle X,Y\rangle_{L^2}  + (t-1)\sum_{k} \lambda_0^k \left\langle Z^k,X \right \rangle_{L^2} ~;~ Z^k_\# P \preceq \nu_0^k \right\}\,.
\end{equation}

For a given $\nu\in\mc{P}_2(\mathbb{R}^d)$, consider the set 
\[
\{ Z \in L^2 ~;~ Z_\# P \preceq_C \nu\} .
\]
One can check that this is a convex, bounded and weakly closed subset of $L^2$. In fact, by Lemma \ref{lem:mcov}, this set is necessarily convex and strongly closed, and hence weakly closed by \cite[Theorem 3.7]{brezis2011functional}, for example; whereas boundedness comes from the fact that $\nu \in \mc{P}_2(\mathbb{R}^d)$. Moreover, the objective function in \eqref{eq:saddle} 
is convex and weakly l.s.c.\ in $X$, and linear in $Z^k$. As a consequence we can apply a standard minimax duality result (see, e.g., Theorem 3.1 in \cite{simons2006minimax}) to exchange the infimum and supremum in \eqref{eq:saddle}. This leads us to the following formulation (after minimization over $X$ and a sign change):
\begin{equation} \tag{$\mc{Q}$} \label{eq:Kto1proj}
    \inf_{Z^k} \left\{ \frac{1}{2} \Big\|tY - (t-1) \sum_k \lambda^k_0 Z^k\Big\|^2_{L^2} ~ ; ~ Z^k_\# P \preceq_C \nu_0^k  \right\}\,.
\end{equation}

\begin{remark}
Note that one would get this same formulation by applying Toland duality to problem \eqref{eq:Kto1}, but viewed as a difference of the convex functions $X \mapsto \|X\|^2/2 + (t-1) \sum_k \lambda_0^k F_{\nu_0^k}(X)$ and $X \mapsto t F_{\nu_1}(X)$; see Section \ref{sec:toland}. 
 \end{remark}

The computations above give us directly the following duality result:

\begin{theorem} \label{th:kto1proj} There holds:
\[
\eqref{eq:Kto1} = - \eqref{eq:Kto1proj} + t \frac{M_2(\nu_1)}{2} -(t-1) \sum_k \lambda_0^k \frac{M_2(\nu_0^k)}{2}.
\]
Moreover $X$ solves \eqref{eq:Kto1} if and only if 
\begin{equation}\label{eq:XYZ}
X = tY - (t-1) \sum_k \lambda^k_0 Z^k
\end{equation}
where $Z^k$ satisfies $Z^k_\# P \preceq_C \nu_0^k$ for all $k=1,\ldots,K$, and $(Z^k)_k$ solve \eqref{eq:Kto1proj}.
\end{theorem}

\begin{remark} Comparing equation \eqref{eq:XYZ} with \eqref{eq:Tfromu} we can deduce that given $Z^k$ solving \eqref{eq:Kto1proj} and $T=\nabla (u^* \circ D)$ constructed as in Theorem \ref{th:duality} form any solution of \eqref{eq:dualBreniermm}, we necessarily have
\begin{equation}\label{eq:zbar}
\overline{Z}\coloneqq \sum_k \lambda^k_0 Z^k = T \circ Y\,.
\end{equation}
This confirms the fact that $T$ must be unique, since problem \eqref{eq:Kto1proj} is strongly convex in $\overline{Z}$. It also gives a more direct characterization of the measure $\overline{\nu}_0 = T_\# \nu_1$ (see Remark \ref{rem:uniqueness}) as the law of $\overline{Z}$, with $Z^k$ solving \eqref{eq:Kto1proj}.
\end{remark}

\begin{remark}[Wasserstein projection in convex order] Let $D^\theta(x) \coloneqq  \theta x$ and $\theta = t/(t-1)$. For 
$K=1$, let us set $\nu_0 = \nu_0^1$ and $Z = Z^1$. Then we obtain that $\overline{\nu}_0 = Z_\# P$ is the unique solution of
\[
\inf \left\{ W^2_2 (D^\theta_\# \nu_1, \overline{\nu}_0) ~;~  \overline{\nu}_0 \preceq_C \nu_0 \right\}\,.
\]
When $K>1$, we do not have a similar $W_2$ projection formulation (e.g., of the type studied in \cite{bolbotowski2025bi}). This is because the constraint in \eqref{eq:Kto1proj} cannot be reduced to a constraint on the law of $\overline{Z}$ defined in \eqref{eq:zbar}.
\end{remark}

\begin{remark}[Stability]  We remark that from problem \eqref{eq:dualprojection} one can easily derive a $W_2$-Lipschitz stability bound on the solution with respect to $\nu_1$, following \cite{kim2026stability,alfonsi2026wasserstein,di2026stability}. On the other hand, the strong convexity of problem \eqref{eq:Kto1X} (or equivalently the strong total convexity of \eqref{eq:Kto1}) implies local 1/2-Hölder stability in $\nu_0^k$, see \cite{gallouet2025metric,di2026stability}.
\end{remark}

\subsection{Multi-plan Weak OT formulation} We finally derive a weak Optimal Transport formulation of problem \eqref{eq:Kto1proj}. This relies on Strassen's Theorem \cite{strassen1965existence}, which states that a convex order relation between two measures is equivalent to the existence of a martingale coupling between them. This can be written in the following form:
\begin{lemma}\label{lem:strassen}
\[
Z_\# P \preceq_C \nu \iff \exists \gamma \in \Gamma(P,\nu) \text{ such that } Z(\omega) = \int y \ed \gamma_\omega(y) \quad \text{ for $P$-a.e. $\omega$} \,.
\]
\end{lemma}

Using this characterization, we can replace the minimization over $Z^k$ in \eqref{eq:Kto1proj} with a minimization over couplings, which yields the following formulation

\[
\inf \left\{ \frac{1}{2} \int \Big | t Y(\omega) - (t-1) \sum_k \lambda^k \int z \ed \gamma_\omega^k (z)\Big |^2 \ed P(\omega) ~ ; ~ \gamma^k \in \Gamma(P,\nu_0^k)  \right\}\,.
\]

Given any $\pi^k\in \Gamma(\nu_1,\nu_0^k)$ we can construct an admissible  $\gamma^k$ by setting \[ \ed \gamma^k(\omega,z) = \ed \pi^k_{Y(\omega)}(z) \otimes \ed P(\omega). \] This means that the problem above is bounded from above by
\begin{equation}\label{eq:wotmm} \tag{$\mc{Q}'$}
\inf \left\{ \frac{1}{2} \int \Big | t y - (t-1) \sum_k \lambda^k \int z \ed \pi_y^k (z)\Big |^2 \ed \nu_1(y) ~ ; ~ \pi^k \in \Gamma(\nu_1,\nu_0^k)  \right\}\,.
\end{equation}

On the other hand, fix arbitrary $\gamma^k \in \Gamma(P,\nu_0^k)$, for all $k$. Denoting $\eta = (\mr{Id}, Y)_\# P$, we can write $\ed P(\omega) = \int \ed \eta_y(\omega) \ed \nu_1(y)$ and hence by Jensen's inequality
\begin{multline}\label{eq:jensenwot}
 \frac{1}{2} \int \Big | t Y(\omega) - (t-1) \sum_k \lambda^k \int z \gamma_\omega^k (z)\Big |^2 \ed P(\omega) \\\geq  \frac{1}{2} \int \Big | t \int Y(\omega) \ed \eta_y(\omega) - (t-1) \sum_k \lambda^k \iint z \ed \gamma_\omega^k (z) \ed \eta_y(\omega) \Big |^2 \ed \nu_1 (y)\,.
 \end{multline}
 By construction $\int Y(\omega) \ed \eta_y(\omega) = y$ for $\nu_1$-a.e.\ $y$. Moreover, for all $k$, we can define a coupling $\pi^k \in \Gamma(\nu_1,\nu_0^k)$ as follows
 \begin{equation}\label{eq:pikwotmm}
 \int f(y,z) \ed \pi^k(y,z) = \iiint f(y,z) \ed \gamma^k_\omega(z) \ed \eta_y(\omega)  \ed \nu_1(y)\,, \quad \forall f\in C_b(\mathbb{R}^d\times \mathbb{R}^d)\,.
 \end{equation}
 Then the right-hand side of \eqref{eq:jensenwot} is bounded from below by \eqref{eq:wotmm}.
This means that \eqref{eq:Kto1proj} is equivalent to \eqref{eq:wotmm}. More precisely, we have proved the following:
\begin{proposition}\label{prop:projwot}
There holds: $\eqref{eq:Kto1proj} = \eqref{eq:wotmm}$.  Moreover, if $(\pi^k)_k$ solves \eqref{eq:wotmm} then the maps $Z^k \in L^2$ defined for all $k$ by
\[
Z^k(\omega) \coloneqq \int_y \ed \pi^k_{Y(\omega)}(y) \,, \quad \text{for } P\text{-a.e.}\, \omega \in \Omega\,,
\]
solve \eqref{eq:Kto1proj}. Conversely, for any solution $(Z^k)_k$ of \eqref{eq:Kto1proj}, the couplings $\pi^k$ defined via equation \eqref{eq:pikwotmm} solve \eqref{eq:wotmm}.
\end{proposition}

\begin{remark} The advantage of \eqref{eq:wotmm} compared with the equivalent problems \eqref{eq:Kto1}, \eqref{eq:dualBreniermm} or \eqref{eq:Kto1proj}, is that \eqref{eq:wotmm} is just a quadratic program with linear constraints. One can therefore use off-the-shelf algorithms to solve it numerically. In \cite{gallouet2025metric}, we studied the case $K=1$, and we proposed a numerical scheme based on an entropy-regularized version of this problem. The scheme proposed below in Section \ref{sec:Kto1numerical} can also be obtained in this way (see Remark \ref{rem:convexity}), although we will provide a different derivation based on the primal problem \eqref{eq:Kto1}.
\end{remark}

\section{A selection principle for linearized barycenters} \label{sec:linearized}
In this section we discuss the link between signed barycenters with one positive coefficient and linearized barycenters \cite{merigot2020quantitative,wang2013linear}. Given an absolutely continuous measure $\rho \in \mc{P}^{ac}_2(\mathbb{R}^d)$ and a family of probability measures $(\nu_0^k)_k  \subset \mc{P}_2(\mathbb{R}^d)$, their linearized OT barycenter with base $\rho$ and coefficients $(\lambda^k)_k \subset \mathbb{R}_{>0}$ is the unique measure defined by
\[
\overline{\nu}_\rho = \Big(\sum_k \lambda^k T^k\Big)_\# \rho
\]
where $T^k$ are the unique optimal transport map from $\rho$ to $\nu_0^k$. When $\rho$ is not absolutely continuous, a common setting in applications \cite{wang2013linear,nenna2022transport,tanguy2024computing}, this definition may be generalized using plans. In this case, however, we may not have uniqueness and one needs to adopt a selection rule. Here we show that the limit $t \downarrow 1$ of problem \eqref{eq:Kto1} yields a possible selection. Conversely, this implies that linearized barycenters can be used to approximate the solutions to \eqref{eq:Kto1} when $t$ is close to one. 

We will mostly rely on the Weak OT formulation \eqref{eq:wotmm} of problem \eqref{eq:Kto1}. We start by observing that rearranging the square in \eqref{eq:wotmm} and dividing the objective by $t(t-1)$, we get that problem \eqref{eq:wotmm} is equivalent (up to constant terms) to
\begin{equation}\label{eq:Bproblem}
\inf_{\pi^k\in \Gamma(\nu_1,\nu_0^k)} \left\{ B + \frac{t-1}{2t} M\right\}\,, 
\end{equation}
where
\[
B(\pi^1,\ldots,\pi^K) \coloneqq \sum_k \lambda^k \int |x_1-x_0|^2\ed \pi^k(x_0,x_1)\,,
\]
and 
\begin{equation}\label{eq:M}
M(\pi^1,\ldots,\pi^K) \coloneqq \int  \Big|\sum_k \lambda^k \mathrm{bary}(\pi^k_{x_1})\Big|^2 \ed \nu_1(x_1) \,.
\end{equation}
In particular, denoting by $(\pi^k_t)_k$ any minimizer of \eqref{eq:wotmm} for any $t>1$, we have that 
\begin{equation}\label{eq:bbb}
 B(\pi_t^k) + \frac{t-1}{2t} M(\pi^k_t) \leq  B(\pi^k) + \frac{t-1}{2t} M(\pi^k)\,,
\end{equation}
for any $\pi^k \in \Gamma(\nu_1,\nu^k_0)$. Since for all $k$ and $t>1$, $\pi^k_t \in \Gamma(\nu_1,\nu_0^k)$, there exists a subsequence $t_n \downarrow 1$ and couplings $\pi_*^k \in \Gamma(\nu_1,\nu_0^k)$ such that 
\begin{equation}\label{eq:pilim}
\pi^k_{t_n} \rightharpoonup \pi_*^k  ~\text{ as } n\rightarrow \infty.
\end{equation}
By lower semicontinuity of $B$ with respect to weak convergence, equation \eqref{eq:bbb} implies that for all $k$,  \[
\pi_*^k \in \Gamma^{\mr{opt}}(\nu^k_0,\nu_1)\,,\]
where $\Gamma^{\mr{opt}}(\mu,\nu)$ is the set of optimal transport coupling for the transport from $\mu$ to $\nu$. 

Let us now consider the measure
\begin{equation}\label{eq:nu0t}
\overline{\nu}_0(t) \coloneqq \Big(\sum_k \lambda ^k T_{\pi^k_t} \Big)_\# \nu_1,
\end{equation}
where for any coupling $\pi^k \in \Gamma(\nu_1,\nu_0^k)$,
\begin{equation}\label{eq:Tkdef} \quad T_{\pi^k}(x_1) \coloneqq \mr{bary}(\pi^k_{x_1}) ~\text{ for $\nu_1$-a.e. }x_1.
\end{equation}
By Proposition \ref{prop:projwot} and Theorem \ref{th:kto1proj} (see also Remark \ref{rem:uniqueness}), it is easy to check that $\overline{\nu}_0(t)$ is uniquely defined for all $t>1$. Moreover, the discussion above and a standard selection argument imply the following result:

\begin{proposition} 
Let us define
\begin{equation}\label{eq:nubar}
\overline{\nu} \coloneqq  \arg\min \Big\{ \int |x|^2 \ed \overline{\nu}_0(x) \, ;\;  \overline{\nu}_0 = \Big(\sum_k \lambda^k T_{\pi^k}\Big)_\# \nu_1\,, \,
\pi^k \in \Gamma^{\mr{opt}}(\nu^k_0,\nu_1) \Big\}.
\end{equation}
where $T_{\pi^k}$ is defined as in \eqref{eq:Tkdef}. Then, $\overline{\nu}_0(t) \rightharpoonup \overline{\nu}$ as $t\downarrow 1$. In particular, if 
$\nu_1$ is absolutely continuous, then $\overline{\nu}_0(t)$ converges  to the unique linearized barycenter of $\nu_0^k$ associated with the weights $\lambda^k$ and with base measure $\nu_1$.
\end{proposition}

\begin{proof} 
Keeping the same notation as above, let us further introduce
\[
T^n \coloneqq \sum_k \lambda^k T_{\pi^k_{t_n}} \,, \quad T^* \coloneqq \sum_k \lambda^k T_{\pi^k_*}. 
\]
By construction,  $T^n \rightharpoonup T^*$ weakly in $L^2(\nu_1)$ as $n \rightarrow \infty$. We start by showing that $T^*_\# \nu_1$ is a minimizer of \eqref{eq:nubar} and that $T^n \rightarrow T^*$ strongly in $L^2(\nu_1)$.
Using the fact that $B(\pi^k_t) \geq B(\pi^k_*)$ for all $t>1$, equation \eqref{eq:bbb} implies that
\[
B(\pi^k_*) + \frac{t-1}{2t} M(\pi^k_t)  \leq B(\pi^k) +  \frac{t-1}{2t} M(\pi^k).
\]
For any $(\pi^k)_k$ such that $\pi^k \in \Gamma^{\mr{opt}}(\nu_1,\nu_0^k)$, $B(\pi^k_*) = B(\pi^k)$. Hence, for any such $\pi^k$,
\[
M(\pi^k_t) \leq  M(\pi^k)\,.
\]
This implies that
\begin{equation}\label{eq:liminf}
\limsup_{n \rightarrow \infty} \| T^n\|_{L^2(\nu_1)}^2 \leq  \| T^*\|_{L^2(\nu_1)}^2
\end{equation}
Furthermore, since $M$ is lower-semicontinuous (this is a consequence of Proposition 9.4 in \cite{gozlan2017kantorovich}) we also have that
\begin{equation}\label{eq:limsup}
 \| T^*\|_{L^2(\nu_1)}^2 \leq \liminf_{n \rightarrow \infty}   \| T^n\|_{L^2(\nu_1)}^2  \leq  M(\pi^k)
\end{equation}
for any $\pi^k \in \Gamma^{\mr{opt}}(\nu_1,\nu_0^k)$.  Hence, $T^*_\# \nu_1$ is a minimizer of \eqref{eq:nubar}. Equations \eqref{eq:liminf} and \eqref{eq:limsup} also imply that $T^n \rightarrow T^*$ strongly in $L^2(\nu_1)$. Hence, $\overline{\nu}_0(t_n) = T^n_\# \nu_1 \rightharpoonup T^*_\# \nu_1$ as $n \rightarrow \infty$.

We conclude by observing that problem \eqref{eq:nubar} admits a unique minimizer by strong convexity with respect to the variable $T = \sum_k\lambda^k T_{\pi^k}$, which is linear in $\pi^k$. This fact implies that $\overline{\nu}_0(t)$ converges without the need of extracting a subsequence.
\end{proof}

\section{General signed barycenters}\label{sec:general}

In this section, we consider the general signed barycenter problem \eqref{eq:signed}. For given $\nu_1^1,\ldots,\nu_1^L \in \mc{P}_2(\mathbb{R}^d)$ and $\nu_0^1, \ldots, \nu_0^K \in \mc{P}_2(\mathbb{R}^d)$, we consider the following problem
\begin{equation}\label{eq:KtoL}
\inf_\mu \left\{ t \sum_{l=1}^L\lambda_1^l\frac{W^2_2(\mu,\nu_1^l)}{2} - (t-1) \sum_{k=1}^K \lambda_0^k \frac{W^2_2(\mu,\nu_0^k)}{2} \right\},
\end{equation}
where $t>1$, $\lambda_0^k,\lambda_1^l>0$ and $\sum_k \lambda_0^l=\sum_l \lambda_1^l =1$. As in the previous section, the choice of the coefficients is not restrictive and will help us draw some analogies with the extrapolation problem. 

\subsection{Existence and non-uniqueness} We start by discussing existence and uniqueness of solutions to \eqref{eq:KtoL}, leveraging the results in \cite{tornabene2024generalized}, and using the same 
point of view of the previous sections. 

\begin{lemma} \label{lem:KtoLexistence} Problem \eqref{eq:KtoL} admits at least one solution. Moreover, $\mu$ is a solution if and only if $\mu = X_\# P$ where $X$ solves
\begin{equation}\label{eq:KtoLX}
\inf_{X} \left\{ \frac{1}{2} \|X\|^2_{L^2} -t \sum_l \lambda_1^l F_{\nu_1^l}(X)  + (t-1) \sum_{k}\lambda_0^k F_{\nu_0^k}(X) \right\}\,.
\end{equation}
\end{lemma}
\begin{proof} The equivalence of problem \eqref{eq:KtoL} with \eqref{eq:KtoLX} follows from expanding the square in the definition of $W_2^2$ and using the definition of $F_\nu$, as in the previous section. 
Existence of solutions for problem \eqref{eq:KtoL} has been established in \cite[Theorem 1.1]{tornabene2024generalized}, and directly implies existence of solutions to \eqref{eq:KtoLX}. We remark that it is not easy to deduce existence for \eqref{eq:KtoLX} directly as in Lemma \ref{lem:Kto1existence}, as the objective is not weakly l.s.c.\ due to the terms associated with the measures $\nu_1^l$.
\end{proof}

The main difference between \eqref{eq:KtoLX} and \eqref{eq:Kto1X} is that the function minimized in \eqref{eq:KtoLX} is not convex. This is because we cannot replace all the $F_{\nu_1^l}$ by inner products $\langle Y^l,\cdot\rangle $ with $Y^l\sim \nu_1^l$ being fixed random variables. 
As a consequence, we cannot expect minimizers to be unique in general. In fact, \cite{tornabene2024generalized} provided a counterexample with two positive terms and all the measures are discrete. Here, we show that non-uniqueness may arise even if all the measures are absolutely continuous (in sharp contrast with classical barycenters with positive coefficients, in which case a single absolutely continuous measure guarantees absolute continuity of the barycenter \cite{agueh2011barycenters}). The configuration we use is illustrated in Figure \ref{fig:nonunique} and is inspired by that used in \cite{tornabene2024generalized} (see also the examples \cite{bolbotowski2025bi} for  related constructions using discrete measures). This result is proved in Appendix \ref{app:counter} using the characterization of minimizers discussed in the next section.  Note that the same counterexample shows that the set of minimizers may not be convex with respect to linear interpolations as it is the case for classical Wasserstein barycenters.

\begin{remark} In \cite{jacobs2026signed} the authors derive a saddle point formulation of the signed barycenter problem (different from the one we derive below in Section \ref{sec:saddle}) and show that if a saddle point exists and one of the measures with positive coefficient is absolutely continuous, then the barycenter is unique. The example in Figure \ref{fig:nonunique} (see also Appendix \ref{app:counter}) implies that one cannot expect such saddle points to exist in general.    
\end{remark}

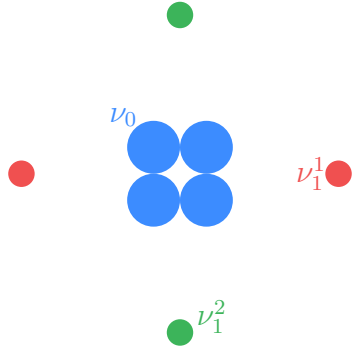
\begin{figure}
\begin{tikzpicture}[scale=.7] 
\definecolor{softblue}{RGB}{60,140,255}
\definecolor{softred}{RGB}{240,80,80}
\definecolor{softgray}{RGB}{60,180,90}

\foreach \x/\y in {0.5/0.5,0.5/-0.5,-0.5/-0.5,-0.5/0.5} {
    \fill[softblue, opacity=1] (\x,\y) circle (0.5);
}

\node[softblue] at (-1.07,1.07) {$\nu_0$};

\fill[softred, opacity=1] (3,0) circle (0.25);
\fill[softred, opacity=1] (-3,0) circle (0.25);
\fill[softgray, opacity=1] (0,3) circle (0.25);
\fill[softgray, opacity=1] (0,-3) circle (0.25);

\node[softred] at (2.48,0) {$\nu_1^1$};
\node[softgray] at (.6,-2.7) {$\nu_1^2$};

\end{tikzpicture}
\caption{We show that for
this configuration for $t\gg 1$ the solution to the signed barycenter with $\lambda_1^1=\lambda_1^2 = 1/2$ cannot not be invariant by rotations of $\pi/2$ and hence it is not unique.
}\label{fig:nonunique}
\end{figure}

\subsection{Necessary optimality conditions} The following proposition provides a characterization of signed barycenters in terms of the $K$-to-1 extrapolation problem. We will use this characterization to construct a counterexample to uniqueness in the case of absolutely continuous data (see Figure \ref{fig:nonunique} and Appendix \ref{app:counter}). Such a  characterization will also be useful in Section \ref{sec:signeddc} for the construction of an algorithm to solve problem \eqref{eq:KtoL} numerically. 

\begin{proposition}\label{prop:crit} Let $X^* \sim \mu$ solve \eqref{eq:KtoLX}. Then for all $Y^l \sim \nu_1^l$ with $l=1,\ldots,L$, such that $(X^*,Y^l)_\# P$ is optimal for the transport from $\nu_1^l$ to $\mu$, $\mu$ is the unique solution of problem \eqref{eq:Kto1} with 
\begin{equation}\label{eq:nu1}
\nu_1 \coloneqq \overline{Y}_\# P\,, \qquad \overline{Y} \coloneqq \sum_{l=1}^L \lambda_1^l Y^l
\end{equation}
\end{proposition}
\begin{proof} Let $Y^l$ and $\overline{Y}$ be defined as in the statement. In view of Lemma \ref{lem:Kto1existence}, we just need to show that $X^*$  minimizes
\[
G(X) \coloneqq \frac{1}{2} \|X\|^2_{L^2} -t \langle X, \overline{Y}\rangle_{L^2}  + (t-1) \sum_{k}\lambda_0^k F_{\nu_0^k}(X).
\]
This is a direct consequence of the definition of $F_{\nu_1^l}$, which implies that for all $X \in L^2$
\[
 \sum_{l=1}^L \lambda_1^l F_{\nu_1}^l(X) \geq  \sum_{l=1}^L \lambda_1^l \langle X, {Y}^l\rangle   = \langle X,\overline{Y}\rangle \,.
\]
Using the fact that $X^*$ is a solution, this inequality implies that $G(X) \geq G(X^*)$ for all $X\in L^2$, and we are done.

\end{proof}

\begin{remark} The definition of $\nu_1$ in \eqref{eq:nu1} and the optimality of the plans $(X^*,Y^l)_\# P$ are equivalent to saying that $\nu_1$ is a linearized Wasserstein barycenter of $\nu_1^1, \ldots, \nu_1^L$ with base $\mu$. Proposition \ref{prop:crit} says that any solution to \eqref{eq:KtoL} is then a composition of a linearized barycenter and a $K$-to-1 extrapolation.
\end{remark}

\subsection{A saddle point formulation} \label{sec:saddle} The result of the previous section can also be seen through the lenses of Toland duality. To see this, we first use Lemma \ref{lem:mcov} to write \eqref{eq:KtoLX} as follows: 
\[
\inf_X \inf_{Z_1^l} \left\{\frac{1}{2} \|X\|^2_{L^2}
-t \Big\langle X, \sum_l \lambda_1^l Z_1^l \Big\rangle_{L^2} + (t-1) \sum_k \lambda_0^k F_{\nu_0^k}(X) ~;~ (Z_1^l)_\# P \preceq_C \nu_1^l   \right\}\,.
\]
Swapping the two infima, and then proceeding as in Section \ref{sec:kto1}, we find that this is equivalent to the following saddle point problem:
\begin{equation}\label{eq:saddle2}
\inf_{Z_1^l} \sup_{Z_0^k} \left\{ - \Big\| t \sum_l \lambda_1^l Z_1^l - (t-1) \sum_k \lambda_0^k Z_0^k \Big \|^2_{L^2} ~;~ (Z_1^l)_\# P \preceq_C \nu_1^l \,, ~(Z_0^k)_\# P \preceq_C \nu_0^k  \right\}\,.
\end{equation}
This is precisely the Toland dual of problem \eqref{eq:KtoLX} (viewed as a different of convex functions).
Note that in this problem the objective is not convex in $Z_1^l$ so here we cannot swap the infimum and the supremum. 

Now, suppose now that $X^*$ is a solution and let $Y^l \sim \nu_1^l$ be such that 
\[
F_{\nu_1^l}(X^*) = \langle X^*, Y^l\rangle_{L^2} \,,
\]
or in other words $Y^l \in \partial F_{\nu_1^l}(X^*)$.
Then
\[
\begin{aligned}
\eqref{eq:KtoLX} &= \frac{1}{2} \|X^*\|^2_{L^2}
-t \Big\langle \sum_l \lambda_1^l Y^l,X^* \Big\rangle_{L^2} + (t-1) \sum_k \lambda_0^k F_{\nu_0^k}(X^*)
\\&\geq 
\inf_X \left\{\frac{1}{2} \|X\|^2_{L^2}
-t \Big\langle \sum_l \lambda_1^l Y^l,X \Big\rangle_{L^2} + (t-1) \sum_k \lambda_0^k F_{\nu_0^k}(X)\right\}
\\&\geq 
\sup_{Z_0^k} \left\{ - \Big\| t \sum_l \lambda_1^l Y^l - (t-1) \sum_k \lambda_0^k Z_0^k \Big \|^2_{L^2} ~; ~(Z_0^k)_\# P \preceq_C \nu_0^k  \right\} \geq \eqref{eq:saddle2}\,.
\end{aligned}
\]
But since \eqref{eq:KtoLX}=\eqref{eq:saddle2}, $Z_1^l = Y_1^l$ solves \eqref{eq:saddle2} and $X_\#^* P$ solves problem \eqref{eq:Kto1} with the data given by equation \eqref{eq:nu1}, so that we recover the result of Proposition \ref{prop:crit}.

\begin{remark}
Clearly, one can also write a saddle point formulation in terms of Brenier potentials, analogous to \eqref{eq:dualBreniermm}, by using Toland duality with respect to the linear structure on $\mc{P}_2(\mathbb{R}^d)$ as explained in Section \ref{sec:toland}. In fact, problem \eqref{eq:KtoL} is equivalent to minimizing the difference of the convex functions:
\[
\mu \mapsto \int \frac{|x|^2}{2}\ed \mu - t \sum_l \lambda^l_1 \operatorname{MCov}(\mu,\nu_1^l) \quad \text{and} \quad \mu \mapsto -(t-1)\sum_k \lambda_0^k \operatorname{MCov}(\mu,\nu_0^k).
\]
Formally, this is then equivalent to minimizing the difference of their Legendre transforms over convex functions $u:\mathbb{R}^d\rightarrow \mathbb{R}$, which are given by the inf convolution of the Legendre transform of each term:
\[
u \mapsto \sup \left\{ -t \sum_l \lambda_1^l\int v_l^* \ed \nu_1^l \;;\; t\sum_l \lambda_1^l v_l  + \frac{|x|^2}{2} =u \right\}
\]
and
\[
u \mapsto \sup \left\{ -(t-1) \sum_l \lambda_0^k \int u_k^* \ed \nu_0^k \;;\; t\sum_l \lambda_0^k u_k =u \right\}\,,
\] 
respectively. Note that in the $K$-to-1 case, we retrieve \eqref{eq:dualBreniermm} from this formulation by setting $w_k = u_k^*$ and then using equation \eqref{eq:ustarD}.
We do not detail this derivation here as such a dual formulation will not be used in the following. 
\end{remark}

\begin{remark}[Signed population barycenters] The results of Section \ref{sec:kto1} on the $K$-to-1 case and those in this section could be directly generalized to the case where the data of the problem, $(\lambda_0^k,\nu_0^k)_k$ and $(\lambda_1^l,\nu_1^l)_l$, is replaced by more general measures 
$\mathbb{P}_0,\mathbb{P}_1 \in \mc{P}_2(\mc{P}_2(\mathbb{R}^d))$. This leads to a population version of the signed barycenter problem:
\[
\inf_\mu \left\{ t \int \frac{W^2_2(\mu,\nu_1)}{2} \ed \mathbb{P}_1(\nu_1) - (t-1) \int \frac{W^2_2(\mu, \nu_0)}{2} \ed \mathbb{P}_0(\nu_0)\,\right\}.
\]
We focused on the discrete setting to simplify the exposition and because it is the most relevant to computations. 
\end{remark}

\section{Numerical approaches}\label{sec:numerical}
In this section we discuss the numerical computation of signed barycenters. All the schemes presented in this section are based on entropic regularization, which we briefly recall in Section \ref{sec:entropic}. While in standard OT this is usually employed with fixed grids, because of the structure of the signed barycenter problem, it is natural to take as main unknown the discrete support of the unknown barycenter. For the $K$-to-1 extrapolation problem,  this can be done without breaking the convexity of the problem. 

\subsection{Entropic regularization and debiasing} \label{sec:entropic}  Let $a = (a_1,\ldots,a_N) \in \mathbb{R}^N_{>0}$ and  $b = (b_1,\ldots,b_M) \in \mathbb{R}^M_{>0}$ two vector of weights such that $\sum_i a_i = \sum_j b_j =1$.
Given $X = (x_1,\ldots,x_N) \in \mathbb{R}^{dN}$ and $Y = (y_1,\ldots,y_M) \in \mathbb{R}^{dM}$, consider the discrete probability measures
\[
\mu = \sum_{i=1}^N a_i \delta_{x_i},
\qquad
\nu = \sum_{j=1}^M b_j \delta_{y_j}\,.
\]
The set of couplings between $\mu$ and $\nu$ can be identified with
\[
\Gamma(a,b)
\coloneqq 
\left\{
\gamma \in \mathbb{R}_{\geq 0}^{M \times N} :
\sum_{j=1}^M \gamma_{ij} = a_i,\;
\sum_{i=1}^N \gamma_{ij} = b_j
\right\}\,.
\]
The relative entropy of $\gamma$ with respect to the product measure $\mu \otimes \nu$ is given by
\[
H(\gamma|\mu \otimes \nu) \coloneqq \sum_{ij} \left[ \log\left(\frac{\gamma_{ij}}{a_i b_j}\right) -1\right] \gamma_{ij}\,.
\]

The entropic \(W_2\) transport problem is defined by adding the relative entropy as a regularization term. Specifically, for a given regularization parameter \(\varepsilon>0\), we set
\begin{equation}\label{eq:w2eps}
 {\mr{OT}}_{\varepsilon}(\mu,\nu) = \inf_{\gamma \in \Gamma(a,b)} \left\{ \sum_{ij} |x_i-y_j|^2 \gamma_{ij} + \varepsilon H(\gamma|\mu \otimes \nu)\right\} \,.
\end{equation}
The main advantage of adding the relative entropy is that this gives rise to a simple and efficient algorithm to compute solutions numerically, generally referred to as Sinkhorn algorithm. This can be seen as an alternate maximization scheme on the dual problem to \eqref{eq:w2eps}, which is given by an optimization problem over potentials $\varphi \in \mathbb{R}^M$, $\psi \in \mathbb{R}^N$,
\begin{equation}\label{eq:dualoteps}
 \frac{{\mr{OT}}_{\varepsilon}(\mu,\nu)}{2} = \sup_{\varphi,\psi}\left\{ \sum_{ij}\left[ - \varepsilon\exp\left( - \frac{|x_i -y_j|^2}{2\varepsilon} + \frac{\varphi_i +\psi_j}{\varepsilon} \right) + \varphi_i  + \psi_j \right]a_i b_j  \right\}\,.
\end{equation}

One of the main drawbacks of $\mr{OT}_{\varepsilon}$ is that it is not a squared distance, and in particular we have $\mr{OT}_{\varepsilon}(\mu,\mu)>0$. To fix this issue, \cite{genevay2018learning} introduced a debiased version which they named Sinkhorn divergence. This is defined as follows:
\[
\mr{S}_\varepsilon(\mu,\nu) \coloneqq  \mr{OT}_\varepsilon(\mu,\nu) - \frac{ 1}2\mr{OT}_\varepsilon(\mu,\mu)  - \frac{ 1}2 \mr{OT}_\varepsilon(\nu,\nu) \,.
\]
Clearly, $\mr{S}_\varepsilon(\mu,\mu) =0$. More precisely, $\mr{S}_\varepsilon$ is a positive-definite and jointly convex function on probability measures \cite{feydy2019interpolating}. Note, however, that here convexity refers to the linear structure on probability measures.

We now collect some further important properties of $\mr{OT}_\varepsilon$ and $\mr{S}_\varepsilon$ that will be useful in the following.
Let us define the weighted norm
\[
\|X\|_a^2 \coloneqq \sum_i|x_i|^2 a_i\,.
\]
The function 
\[
X \mapsto \frac{\|X\|^2_a}{2} -\frac{1}2 \mr{OT}_\varepsilon\left( \sum_i a_i \delta_{x_i}, \nu \right)
\]
is  convex, since by definition of $\operatorname{OT}_\varepsilon$, it can be expressed as a supremum of linear functions in $X$. This generalizes the fact that $F_\nu$ is convex to the entropic case. 
Clearly, to establish a similar result for the Sinkhorn divergence we need to determine whether the entropic self-bias $X\mapsto \mr{OT}_\varepsilon(\mu,\mu)$ is semi-convex (uniformly in $\varepsilon$). This is established in the following lemma.

\begin{lemma}\label{lem:convdebias}
The function
\[
X \mapsto \alpha \frac{\|X\|^2_a}{2} + \frac{1}{2}\mr{OT}_\varepsilon\left( \sum_i a_i \delta_{x_i},\sum_i a_i \delta_{x_i}\right)
\]
is convex for $\alpha\geq 4 e^{-3/2}$.
\end{lemma}

\begin{proof}
By symmetry the optimal potentials in the dual problem defining $\mr{OT}_\varepsilon$ coincide, i.e., $\varphi=\psi$. By the optimality conditions of the dual problem, the optimal $\varphi$ solves 
\[
{\varphi_i} = -\varepsilon \log\sum_j \exp\left( -\frac{|x_i-y_j|^2 -2\varphi_j}{2\varepsilon} \right)\,.
\]
By standard properties of the $\mr{LogSumExp}$ function, we obtain that 
\[
{\varphi_i} \leq \inf_j \left\{\frac{|x_i-x_j|^2 }{2}- \varphi_j \right\} \leq -{\varphi_i}\,,
\]
and hence $\varphi_i\leq 0$. Since $\varphi$ is non-positive, we can write $\mr{OT}_\varepsilon(\mu,\mu)/2$ as a supremum of functions of the form
\[
- \varepsilon \sum_{ij} \exp\left( - \frac{|x_i -x_j|^2}{2\varepsilon} \right) m_im_j  
\]
with $m_i \coloneqq \exp(\varphi_i/\varepsilon)a_i \in[0,a_i]$. We conclude by observing that the function $z\mapsto -\varepsilon \exp(-|z|^2/(2\varepsilon))$ is ($-2e^{-3/2}$)-convex.
\end{proof}

\begin{remark}
     Note that this implies semi-concavity (uniform in $\varepsilon$) of the Sinkhorn divergence in terms of the particle positions. A similar result was proven in \cite{carlier2024displacement} for displacement convexity, but requires a compact domain. 
\end{remark}

\subsection{Computation of $K$-to-1 signed barycenters}\label{sec:Kto1numerical} Let us consider the $K$-to-1 extrapolation problem \eqref{eq:Kto1} with data 
\[
\nu_0^k = \sum_{i=1}^{N^k_i} a_i^k \delta_{x_i^k} \,, \quad \nu_1 = \sum_{j=1}^M b_j \delta_{y_j}\,.
\]
By the second point of Theorem \ref{th:duality} (see also Remark \ref{rem:uniqueness}), the unique solution to the problem will necessarily be in the form 
\[
\mu = \sum_{j=1}^M b_j \delta_{z^j}\,.
\]
Hence we can define a discrete solution by finding $Z \in \mathbb{R}^{dM}$ solving
\begin{equation}\label{eq:Kto1eps}
\inf_Z \left\{ t \frac{\|Z-Y\|^2_b}{2} -  \frac{(t-1)}2 \sum_k \lambda_0^k \,\mr{OT}_\varepsilon \left(\sum_{j=1}^M b_j \delta_{z_j},\nu_0^k\right) \right\}\,.
\end{equation}
By the strong convexity of the objective function in $Z$,  this problem has a unique solution $Z^\varepsilon$ for all $\varepsilon>0$, and moreover
\[
\sum_j b_j \delta_{z_j^\varepsilon} \rightharpoonup \mu\,,
\]
where $\mu$ is the unique solution of \eqref{eq:Kto1}. 

Replacing $\mr{OT}_\varepsilon$ with its dual formulation in \eqref{eq:dualoteps}, we obtain the equivalent problem
\begin{equation}\label{eq:inff}
\inf_{Z, \bs{\varphi},\bs{\psi}} f(Z,\bs{\varphi},\bs{\psi})
\end{equation}
where $\bs{\varphi} = (\varphi^1,\ldots,\varphi^K) \in \mathbb{R}^{N^1} \times \ldots\times \mathbb{R}^{N^K}$, $\bs{\psi} = (\psi^1,\ldots,\psi^K) \in (\mathbb{R}^{M})^K$, and
\begin{multline}
f(Z, \bs{\varphi}, \bs{\psi}) \coloneqq  t \frac{\|Z-Y\|^2_b}{2}\\ -  (t-1) \sum_{ijk} \lambda_0^k\left[ - \varepsilon\exp\left( - \frac{|x^k_i -z_j|^2}{2\varepsilon}  + \frac{\varphi^k_i + \psi^k_j}{\varepsilon} \right) + \varphi_i^k  + \psi_j^k \right]a_i^k b_j \,.
\end{multline}

\begin{remark}[Convexity of the discrete problem] \label{rem:convexity} Note that problem \eqref{eq:inff} is a (jointly) convex optimization problem; see Proposition 6.1 in \cite{gallouet2025metric}. In fact, just as for the metric extrapolation problem, \eqref{eq:inff} can also be obtained directly as the dual of the entropy-regularized version of the corresponding barycentric problem \eqref{eq:wotmm}.
\end{remark}

\subsubsection{SISTA}\label{sec:sista} Problem \eqref{eq:inff} can be solved using the following iterative procedure: Given $\tau>0$, $\bs{\psi}^0$, $Z^0$, for $n\geq 0$
\[\begin{array}{l}
\bs{\varphi}^{n+1} = \displaystyle \mr{arg} \min_{\bs{\varphi}} f(Z^n,\bs{\varphi},\bs{\psi}^n)\,,\\
\bs{\psi}^{n+1} =\displaystyle \mr{arg} \min_{\bs{\psi}} f(Z^n,\bs{\varphi}^{n+1},\bs{\psi})\,,\\
Z^{n+1} = Z^n - \tau \nabla_Z f(Z^n,\bs{\varphi}^{n+1},\bs{\psi}^{n+1})\,.
\end{array}
\]
The updates of each component of $\bs{\varphi}$ and $\bs{\psi}$ are independent and correspond to one iteration of Sinkhorn's algorithm. They are followed by an explicit Euler step for the minimization over $Z$ with time step $\tau$ (exact minimization would be impractical here since the minimizer does not have an explicit form). The convergence of the algorithm was studied in \cite{gallouet2025metric} in the case $K=1$, in which case the algorithm converges with linear rate as long as $\tau \leq C \varepsilon$, where $C$ is an explicit constant independent of $\varepsilon$. Although we do not do it here for the sake of brevity, it is not difficult to generalize this analysis to the case $K>1$.

\subsubsection{Debiased $K$-to-1 extrapolation}\label{sec:debkto1}
We can define the debiased $K$-to-1 extrapolation in an analogous way, by using $\mr{S}_\varepsilon$ instead of $\mr{OT}_\varepsilon$. This yields the problem
\begin{equation}\label{eq:Kto1debiased}
\inf_Z \left\{ t \frac{\|Z-Y\|^2_b}{2} -  \frac{(t-1)}2 \sum_k \lambda_0^k \,\mr{S}_\varepsilon \left(\sum_{j=1}^M b_j \delta_{z_j},\nu_0^k\right) \right\}\,.
\end{equation}
By Lemma \ref{lem:convdebias}, this problem is convex as long as
\begin{equation}\label{eq:tstar}
1 - (t-1)\alpha > 0\quad \implies t < t^* \coloneqq \frac{1+\alpha}{\alpha}\,,
\end{equation}
with $\alpha \geq 4 e^{-3/2}$ as in Lemma \ref{lem:convdebias}. Problem \eqref{eq:Kto1debiased} can be solved numerically using the same strategy described above as long as $t< t^*$. While debiasing reduces the convexity of the problem, it improves the accuracy of the solution. This is demonstrated by the example below.

\begin{example} The metric extrapolation from $\mu$ to itself is always equal to $\mu$. Howeover, one can easily check that this is no more the case in the entropy regularized setting. We show here that at least in a very simple setting, using $\mr{S}_\varepsilon$ instead of $\mr{OT}_\varepsilon$ solves this issue. To show this, let us first compute $\mr{OT}_\varepsilon(\mu,\mu)$ in the case where $\mu = (\delta_{-p/2} + \delta_{p/2})/2$. By symmetry,  the density of the optimal plan with respect to $\mu\otimes \mu$ is given by
\[
\rho(x,y) = \frac{4}{Z} \exp \left(-\frac{|x-y|^2}{2\varepsilon}\right)
\]
where $Z = 2 \exp(-|p|^2/(2\varepsilon)) + 2$.
Then
\[
\mr{OT}_{\varepsilon}(\mu,\mu) = -\varepsilon \log(Z) - \varepsilon
\]
which is semi-convex in $p$ uniformly in $\varepsilon$.
Consider now the extrapolation problem form $\mu$ to $\mu$, regularized using the Sinkhorn divergence as in \eqref{eq:Kto1debiased}. By symmetry, the solution has the form $\mu =  (\delta_{-q/2} + \delta_{q/2})/2$, where $q$ minimizes
\[
\frac{|q-p|^2}{8(t-1)} + \frac{\varepsilon}{t}\log\left( \exp\left(-\frac{|q-p|^2}{8\varepsilon}\right) +  \exp\left(-\frac{|q+p|^2}{8\varepsilon}\right)\right) - \frac{\varepsilon}{2t} \log\left( \exp\left(-\frac{|q|^2}{2\varepsilon}\right) + 1\right)\,. 
\]
By the first order optimality conditions, one can verify that $q =p$ is a critical point for all $\varepsilon>0$ (it is the unique minimizer as long as $t$ is sufficiently small so that the problem stays convex, but it corresponds to a local maximum for $t$ large).
\end{example}

\subsubsection{A numerical example} In Figure \ref{fig:Kto1} we provide a numerical example for the solution of $K$-to-1 signed barycenters, obtained using the algorithm discussed in Section \ref{sec:sista}, without debiasing and for $K=2$. 
The measures $\nu_0^1$, $\nu_0^2$ and $\nu_1$ are obtained via uniform quantization of given shapes. The measure $\overline{\nu}_0$ corresponding to different values of $t$ and the parameters $\lambda_0^1$ and $\lambda_0^2$ is shown in Figure \ref{fig:linbary}. As discussed in Section \ref{sec:linearized}, as $t\downarrow 1$, $\overline{\nu}_0$ approaches a linearized barycenter of $\nu_0^1$ and $\nu_0^2$, with base $\nu_1$. Note that in this section we report only the results obtained with the biased version of the algorithm, which computes solutions to \eqref{eq:Kto1eps}, as debiasing has only a minor qualitative impact on these examples. The debiased approach of Section~\ref{sec:debkto1} will be needed in the application considered in the next section to reduce the numerical artifacts associated with  entropic regularization.

\begin{figure}[t]
    \centering
    \includegraphics[trim={0.5cm 0.5cm 0cm 0cm},clip,width=0.7\textwidth]{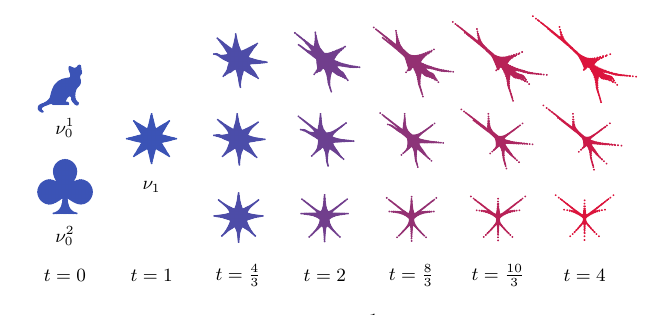}
	\caption{Example of $K$-to-$1$ signed barycenters obtained using the algorithm in Section \ref{sec:sista} with $\varepsilon=10^{-3}$. Top line: $\lambda_0^1=1, \lambda_0^2=0$; middle line:  $\lambda_0^1=\lambda_0^2=\frac{1}{2}$; bottom line: $\lambda_0^1=0, \lambda_0^2=1$.}
	\label{fig:Kto1}
\end{figure}

\begin{figure}[t]
    \includegraphics[trim={0cm 0cm 0cm 0cm},clip,width=0.8\textwidth]{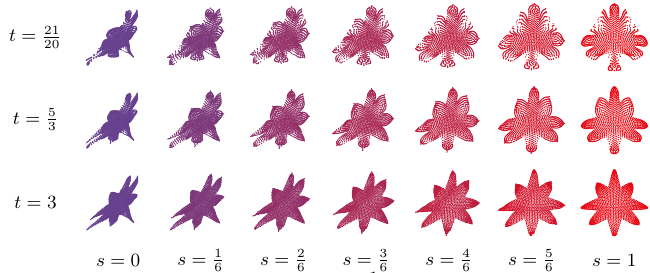}
	\caption{The measure $\overline{\nu}_0$ for $\nu_0^1,\nu_0^2,\nu_1$ as in Figure \ref{fig:Kto1}, for $\lambda_0^1=s$, $\lambda_0^2=(1-s)$ and different values of $t$. For $t\approx 1$, the curve $(\overline{\nu}_0)_{s\in[0,1]}$ approximates the linearized geodesic from $\nu_0^1$ to $\nu_0^2$ with base measure $\nu_1$.}\label{fig:linbary}
\end{figure}

\subsection{Toland duality for quantization and classical barycenters} \label{sec:classical}
The objective of this paragraph is to show how the metric extrapolation problem arises naturally also in the context of the computation of particle approximations of classical Wasserstein barycenters, i.e., when all the coefficients are positive.  
Specifically, we seek a discrete measure supported on $Y \in \mathbb{R}^{dM}$ with given masses $b \in \mathbb{R}^{M}_{>0}$ with $\sum_j b_j = 1$, and approximating the barycenter of $K$ given measures $\nu^1,\ldots,\nu^K$, i.e.,
\begin{equation}\label{eq:barycenter}
 \inf_{Y} \, \sum_{k=1}^K \lambda^k W^2_2\left( \sum_{j=1}^M b_j \delta_{y_j}, \nu^k \right) \,,
\end{equation}
where $\lambda^1,\ldots,\lambda^K>0$ and $\sum_k \lambda^k =1$. Note that when $K=1$, this is a variant of the optimal quantization problem \cite{merigot2021non}.

The connection between this problem and the methods discussed in this work relies on the following proposition which is another application of Toland duality: 

\begin{proposition} \label{prop:barycenterco} $Y$ solves \eqref{eq:barycenter} if and only if it solves
\begin{equation}\label{eq:barycenterco}
\inf \left\{ - \frac{\|Y\|^2_b}{2}~;\quad Y = \sum_k \lambda^k X^k\,, \quad \sum_j b_j \delta_{x_j^k} \preceq_C \nu^k\right\} \,.
\end{equation}
\end{proposition}
\begin{proof} From Lemma \ref{lem:mcov} we can deduce that
\[
F_{\nu^k}(Y) = \sup_{X^k} \left\{ \langle  X^k,Y\rangle_b\,;\, \sum_j b_j \delta_{x_j^k} \preceq_C \nu^k \right\}\,.
\]
Note in particular that we can take $X^k$ discrete (rather than a general random variable) since $Y$ is discrete. Then, up to constant terms, we can write \eqref{eq:barycenter} as follows:
\begin{multline}\label{eq:decomposition}
\inf_Y \left\{ \frac{\|Y\|^2_b}{2} - \sum_k \lambda^k F_{\nu^k}(Y) \right\} \\= \inf_Y \inf_{X^k} \left\{ \frac{\|Y\|^2_b}{2} - \sum_k \lambda^k \langle  X^k,Y\rangle_b\,;\, \sum_j b_j \delta_{x_j^k} \preceq_C \nu^k \right\}\,.
\end{multline}
Minimizing with respect to $Y$ gives the result.
\end{proof}

Let us formulate the proximal gradient method for problem \eqref{eq:barycenterco}. Given $Y^0$, for $n\geq 0$, 
\begin{equation}\label{eq:implicitquant}
Y^{n+1} = \mr{arg}\min_{Y} \left\{ \frac{\|Y- Y^n\|^2_b}{2\tau}- \frac{\|Y\|^2_b}{2}~;~ Y = \sum_k \lambda^k X^k\,,~ \sum_j b_j \delta_{x_j^k} \preceq_C \nu^k\right\}\,.
\end{equation}
The scheme is well-defined for $\tau<1$, in which case every step has a unique solution. It is also easy to see that the sequence $(Y^n)_n$ is bounded and any accumulation point $Y^*$ is a stationary point of \eqref{eq:barycenterco}, i.e., $Y^*$ satisfies the constraint and 
\begin{equation}\label{eq:opty}
\langle Y^*, Y\rangle_b \leq  \| Y^*\|^2_b \qquad \forall \,Y= \sum_k \lambda^k X^k \text{ such that } \sum_j b_j \delta_{x_j^k} \preceq_C \nu^k\,,
\end{equation}
(the proof follows the same lines as Proposition \ref{prop:dcconv}; see Appendix \ref{app:conv}).

Comparing this with problem \eqref{eq:Kto1proj}, we have the following interpretation of the scheme: 
\[
Y^{n+1} = Y^{n} + \frac{\tau}{1-\tau} (Y^n - Z^{n+1})\,,     
\]
where $Z^{n+1}$ is uniquely determined by the conditions
\[
\mu = \sum_j b_j \delta_{z_j^{n+1}}~~ \text{ solves \eqref{eq:Kto1} with } t = \tau^{-1}\,, ~\nu_0^k = \nu^k \text{ and } \nu_1 = \sum_{j} b_j \delta_{y^n_j}\,,
\]
and \[ 
\pi = \sum_j b_j \delta_{(y_j^{n},z_j^{n+1})} \in \Gamma^{\mr{opt}}(\nu_1,\mu).
\]
where $\Gamma^{\mr{opt}}(\mu,\nu)$ denotes the set of optimal couplings for the transport from $\mu$ to $\nu$.

Taking $\tau =1$ in \eqref{eq:implicitquant}, we recover existing fixed point strategies for Wasserstein barycenters studied in \cite{alvarez2016fixed, tanguy2024computing} (see Section \ref{sec:linearized}), or Lloyd's algorithm when $K=1$ \cite{merigot2021non}. The regularization induced by taking $\tau<1$ makes the algorithm easier to analyze for arbitrary data and bypasses the constructions of glued couplings as in \cite{tanguy2024computing}, but it does not affect substantially the numerical results; see Figure \ref{fig:bary}.  Note that for this test we use the debiased version of the algorithm descibed in Appendix \ref{sec:debbary}, which yields solutions to \eqref{eq:debbary}, known as debiased Sinkhorn barycenters \cite{janati2020debiased}. Without debiasing, the effects of entropy regularization are well known to induce significant mass concentration artifacts in the computed barycenters. Alternative regularization approaches can also alleviate this issue \cite{chizat2025doubly}, although they are less straightforward to incorporate into our framework.

\begin{figure}[t]
    \centering
    \begin{subfigure}[c]{0.35\textwidth}
        \centering
        \includegraphics[width=.85\textwidth]{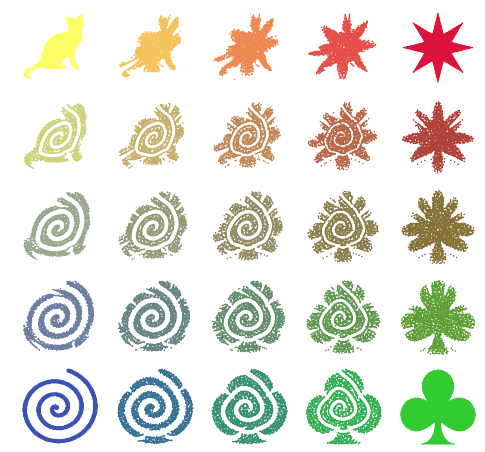}
    \end{subfigure}%
    \begin{subfigure}[c]{.2\textwidth}
        \centering
        \includegraphics[width=\textwidth]{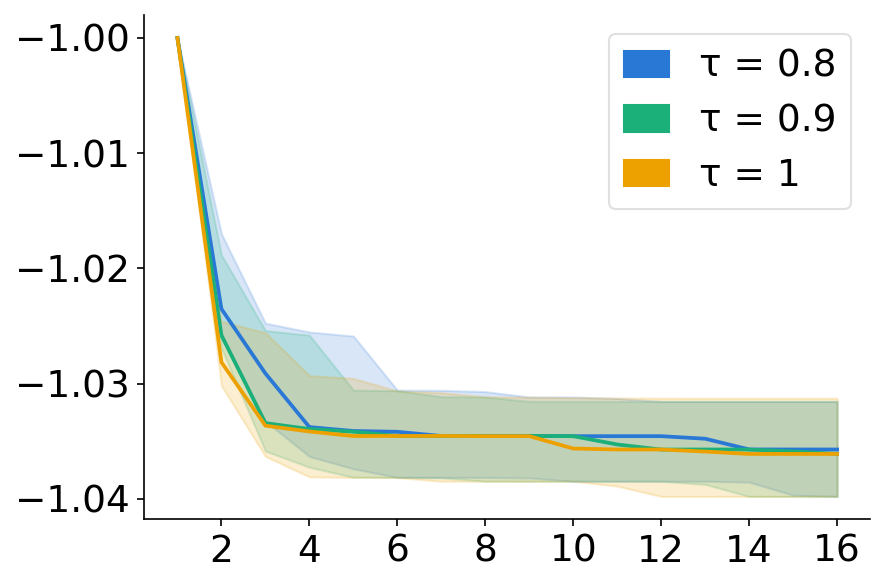}\\[0.3cm]
        \includegraphics[width=\textwidth]{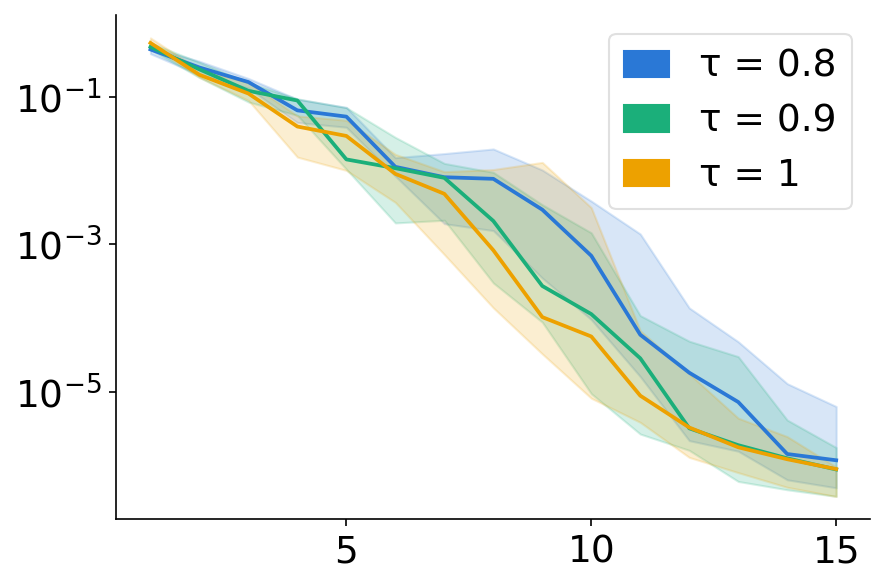}
    \end{subfigure}
    \caption{On the left, barycenters computed with the entropically regularized and debiased version of algorithm \eqref{eq:implicitquant} (see Appendix \ref{sec:debbary}) for $\tau=0.85$ and $\varepsilon=10^{-3}$: the four corner images are interpolated with $(\lambda_i)_{i=1}^4=( (1-s)(1-t), s(1-t), (1-s)t, st )$ for $s,t\in\{0,\frac{1}{4},\frac{1}{2},\frac{3}{4},1\}$. On the right, the convergence results for a different test where $(\nu_0^k)_{k=1}^4$ are drawn from the same normal distribution in dimension $d=2$, and $M=N=30$; the test is repeated $15$ times and the median curve is represented, with the shaded region spanning the 33-66$\%$ quantiles. On the top, comparison in the decrease of the error on the iterate; on the bottom, comparison in the relative decrease of the objective function, $f_i/|f_0|$.}
    \label{fig:bary}
\end{figure}

\subsection{Computation of signed barycenters via a DC method}\label{sec:signeddc}
Let $\nu_0^k, \in \mathcal{P}_2(\mathbb{R}^d)$, $\lambda_0^k>0$ for $k=1,\ldots,K$, and $\nu_1^l \in \mathcal{P}_2(\mathbb{R}^d)$, $\lambda_1^l>0$ for $l=1,\ldots,L$, such that $\sum_l \lambda_1^l =1$, $\sum_l \lambda_1^l =1$ and $t>1$. We seek to compute particle approximations of the signed barycenter problem, defined by the following problem:
\[\inf_{Z} \, \left\{ t \sum_{l=1}^L \frac{\lambda_1^l}{2}  W^2_2\left(\sum_{j=1}^M b_j \delta_{z_j}, \nu_1^l \right) - (t-1)\sum_{k=1}^K \frac{\lambda_0^k}{2} W^2_2\left(\sum_{j=1}^M b_j \delta_{z_j}, \nu_0^k \right)\right\}\,.
\]

Up to constant terms, the problem above can be written as follows:
\begin{equation}\label{eq:dc0}
\inf_{Z} \left\{ \frac{\|Z\|_b^2}{2t} + F_0(Z) - F_1(Z) \right\}
\end{equation}
where
\[
 F_0(Z) \coloneqq \frac{t-1}{t} \sum_{k=1}^K \lambda_0^k \mathrm{MCov}\left( \sum_{j=1}^M b_j \delta_{z_j}, \nu_0^k \right)\,, \quad F_1(Z) \coloneqq \sum_{l=1}^L \lambda_1^l \mathrm{MCov}\left( \sum_{j=1}^M b_j \delta_{z_j}, \nu_1^l \right) \,.
\]

As equation \eqref{eq:dc0}  consists in the minimization of a difference of convex functions, it can be solved using the the following classical alternate optimization approach for DC programming:  Given $Z^0$, for $n\geq 1$,
\begin{equation}\label{eq:dc}
Y^n \in \partial F_1(Z^{n-1}) \,,\quad Z^n \in \partial \tilde{F}_0^*(Y^n) \,, \quad \tilde{F}_0(\cdot) = \frac{\|\cdot\|_b^2}{2} + F_0(\cdot) 
\end{equation}
where $\partial$ denotes the subdifferential with respect to the weighted inner product $\langle X,Y\rangle_b = \sum_i \langle x_i,y_i\rangle b_i$. More explicitly, we have 
\begin{equation} \label{eq:partialF1}
\partial F_1(Z) =  \left\{ \sum_{l=1}^L \lambda_l^+ B^l(Z)\,; ~ B^l(Z)_j =  \int x \mathrm{d} \pi^l_{z_j}(x), ~ \pi^l \in \Gamma^{\mr{opt}}\left( \nu_1^l, \sum_j b_j \delta_{z_j} \right) \right\}
\end{equation}
where $\mathrm{d}\pi^l(y,z) = \sum_j\delta_{z_j} \otimes \mathrm{d} \pi^l_{z_j}(y)$. As for the subdifferential of $\tilde{F}^*$, we first observe that 
\[
\tilde{F}^*_0(Y) = \frac{\|Y\|^2_b}{2} + \sup_Z \left\{-\frac{\|Z -Y\|^2_b}{2} +\frac{t-1}{2t} \sum_{k=1}^K \lambda_0^k W^2_2\left( \sum_j b_j \delta_{z_j} , \nu_0^k\right) \right\}\,.
\]
From this fact and the analysis in Section \ref{sec:kto1} one can deduce the following characterization:

\begin{lemma} $\tilde{F}^*_0 \in C^1$ and $\nabla \tilde{F}^*(Y) = Z^*$ is uniquely determined by the conditions
\[
\mu = \sum_j b_j \delta_{z_j^{*}}~~ \text{ solves \eqref{eq:Kto1} with } \nu_0^k = \nu^k \text{ and } \nu_1 = \sum_{j} b_j \delta_{y_j}\,,
\]
and \[ 
\pi = \sum_j b_j \delta_{(y_j,z_j^*)} \in \Gamma^{\mr{opt}}\left(\nu_1,\mu\right).
\]
\end{lemma}

Hence,  algorithm \eqref{eq:dc} computes alternatively a linearized barycenter (with positive coefficients) of the measures $\nu_1^l$ with base the current estimate of the solution, and then extrapolates from the measures $\nu_0^k$ to such linearized barycenter to compute a new estimate of the solution. For $L=1$, in particular, we recover the simple $K$-to-1 extrapolation problem. For $L>1$, the algorithm can be viewed as a fixed point method on the optimality conditions given in Proposition \ref{prop:crit}. 

A classical convergence result for algorithm \eqref{eq:dc} guarantees that accumulation points of the iterates are stationary points. We recall this result in the following proposition, whose proof is provided in the appendix for completeness. 

\begin{proposition} \label{prop:dcconv} The sequence $(Z^n)_n$ defined by \eqref{eq:dc} is bounded and any accumulation point is a stationary point in the sense of Proposition \ref{prop:crit}.
\end{proposition}

An illustrative example obtained using this scheme, for the same type of data as in the previous sections, is shown in Figure \ref{fig:KtoL}.

\begin{figure}[t]
    \centering
    \includegraphics[trim={0cm 0cm 0cm 0cm},clip,width=0.7\textwidth]{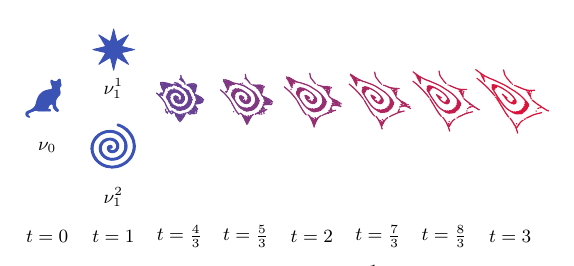}
    \raisebox{1.5cm}{\includegraphics[trim={0cm 0cm 0cm 0cm},clip,width=0.2\textwidth]{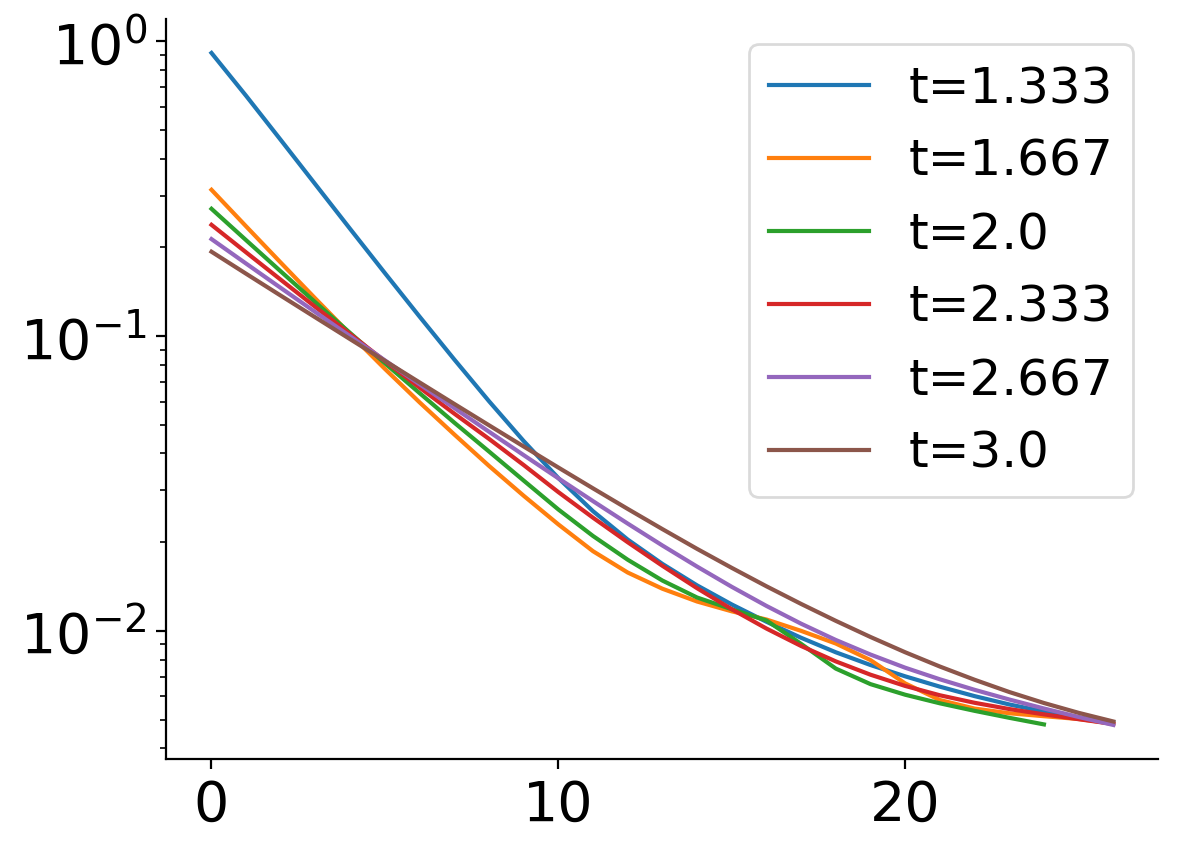}}
	\caption{Example of signed Wasserstein barycenters computed with the entropy-regularized version of algorithm \eqref{eq:dc} ($K=1$, $L=2$, $\lambda_1^1=\lambda_1^2=\frac{1}{2}$) with debiasing (see Appendix \ref{sec:debsignedbary}) and $\varepsilon = 10^{-3}$. On the right, convergence profile for the objective function in \eqref{eq:signed} along the iterations.}
	\label{fig:KtoL}
\end{figure}

\subsection{Global Fréchet regression} Following \cite{petersen2019frechet}, one can use signed barycenters to generalize linear regression to models whose responses take values in metric spaces.  Let $(X,Y)\sim \pi \in \mc{P}(E \times M)$, with $E$ a Euclidean space and $(M,d)$ a metric space. Given i.i.d.\ samples $(X_i,Y_i) \sim \pi$ for $i=1,\ldots,n$,  denote by  $\bar{X} = n^{-1} \sum_{i=1}^n X_i$ and $\hat{\Sigma} = n^{-1}\sum_{i=1}^n(X_i-\bar{X})(X_i -\bar{X})^T$ the sample mean and sample covariance of the predictor, respectively. We define an estimator for the regression function as follows:
\[
\hat{m}(x)=\operatorname{arg}\min_y\frac1n\sum_{i=1}^n\hat{s}(x,X_i)d^2(y,Y_i),
\qquad
\hat{s}(x,z)=1+(x-\bar{X})^T\hat{\Sigma}^{-1}(z-\bar{X}).
\]
provided that the minimizer exists and is unique. When $M$ is a Euclidean space, this reduces to the standard linear regression estimator.  

We consider here the setting where $E=\mathbb{R}$ and $M$ is $\mc{P}_2(\mathbb{R}^d)$ equipped with the $W_2$ distance. A numerical example for this setting is shown in Figure \ref{fig:regression}. Here, the data obtained to compute $\hat{m}(x)$ is obtained by sampling uniformly in time a given time evolution of a discrete measure.  For a given time value $x$, $\hat{m}(x)$ is a signed barycenter, which we computed using the entropy-regularized version of algorithm   \eqref{eq:dc} with debiasing (see Appendix \ref{sec:debsignedbary}) and $\varepsilon = 10^{-3}$. Note that, as one might expect, the predicted solution is qualitatively closer to a Wasserstein geodesic than the true evolution.

We remark that in this general case the lack of uniqueness is an obstacle to apply the theory developed in \cite{petersen2019frechet}. (Sufficient requirements in this direction for the Gaussian case have been recently provided in \cite{nguyen2026fr}). Moreover, the numerical approach we use is only guaranteed to converge to a stationary point; see Proposition \ref{prop:dcconv}. Despite these limitations, the numerical experiment illustrates that the proposed method can still recover an evolution with the expected geometric features. 

Note that other models have been proposed to generalize linear regression to the Wasserstein space. These deal with the issue of finite time existence of geodesics without relying on signed barycenters, e.g., by linearization of the Wasserstein space itself \cite{wang2013linear,seguy2015principal} or by explicit parameterization of the geodesics \cite{vesseron2026wasserstein}. 

\begin{figure}[t]
	\setlength{\tabcolsep}{0pt}
	\scriptsize
	\centering
	\begin{tabular}{C{0.05\textwidth}C{0.22\textwidth}C{0.22\textwidth}C{0.22\textwidth}C{0.24\textwidth}C{0.05\textwidth}}
		& \includegraphics[trim={.5cm .5cm .5cm .5cm},clip,width=0.22\textwidth]{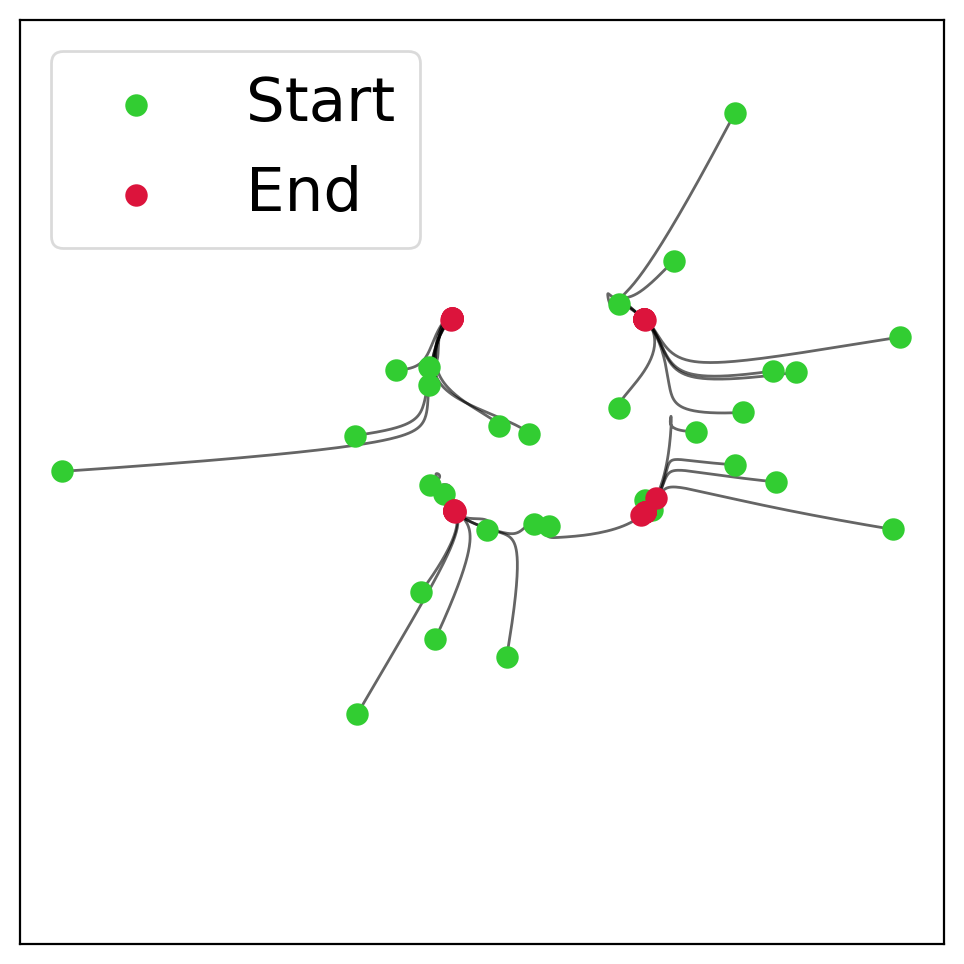} &
		\includegraphics[trim={.5cm .5cm .5cm .5cm},clip,width=0.22\textwidth]{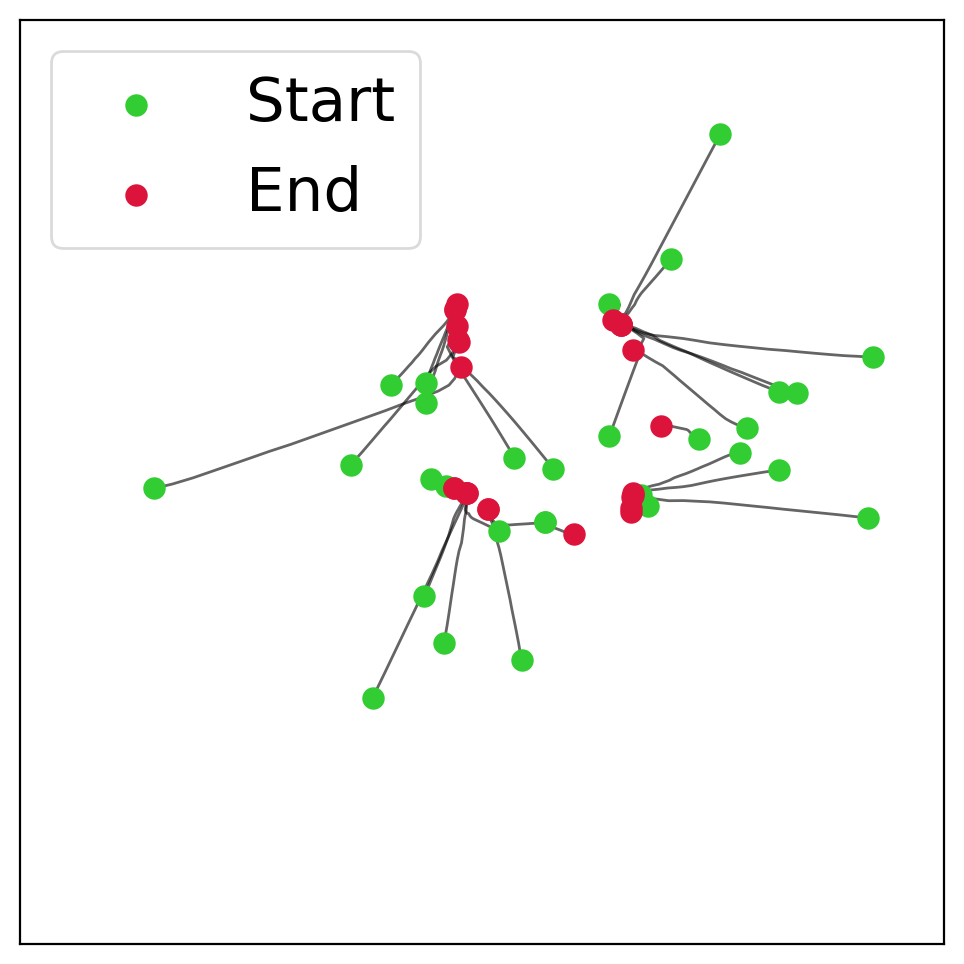} & \includegraphics[trim={.5cm .5cm .5cm .5cm},clip,width=0.22\textwidth]{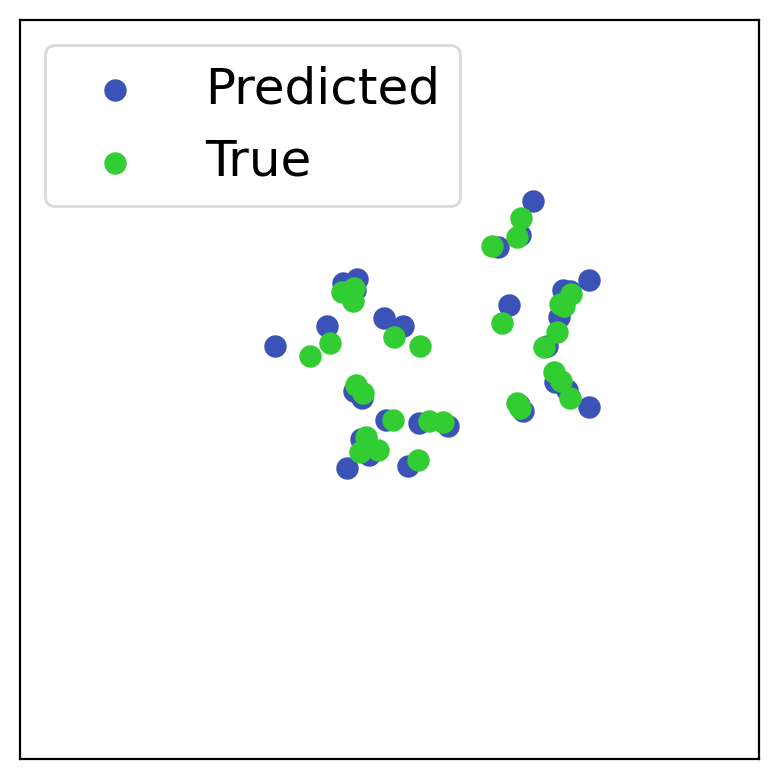} &
	    \raisebox{.8cm}{\includegraphics[trim={0cm 0cm 0cm 0cm},clip,width=0.2\textwidth]{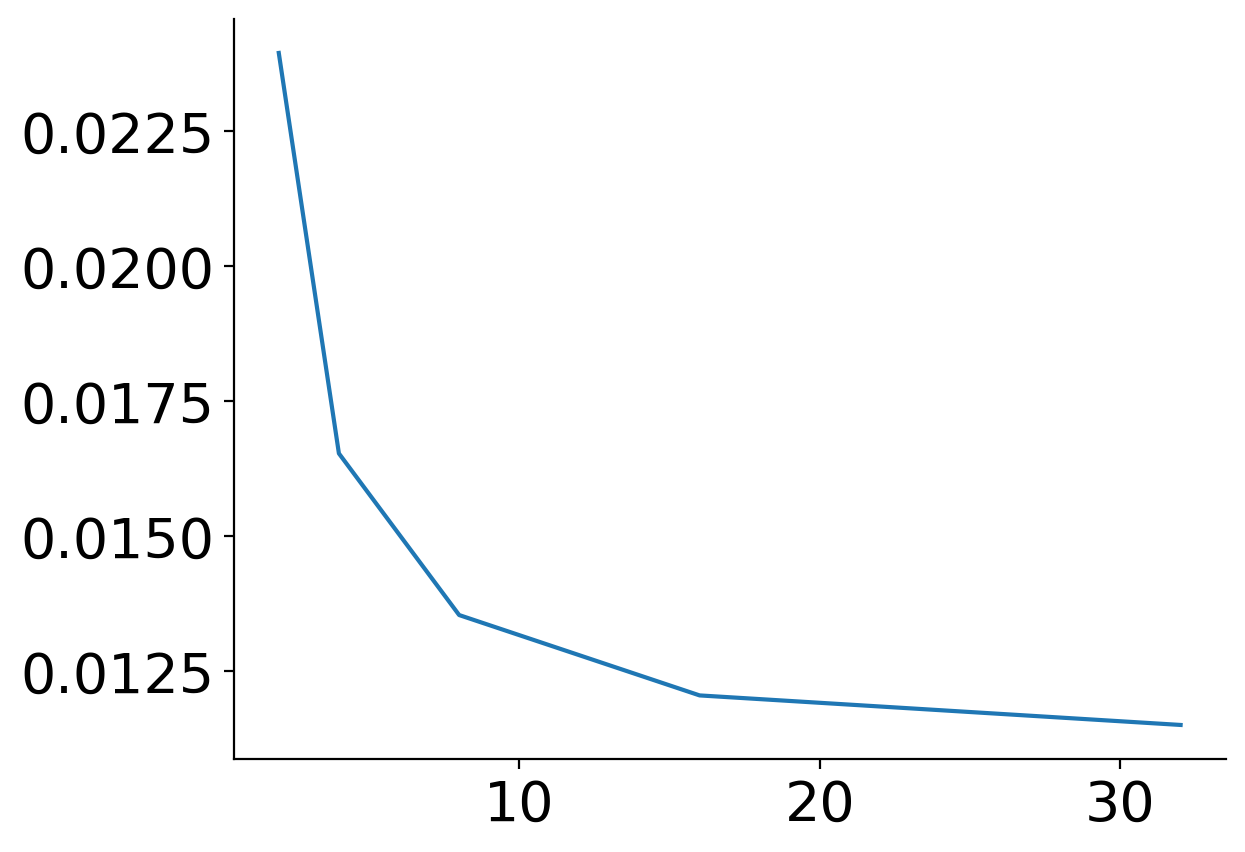}}
	    &
	\end{tabular}\vspace{-2em}
	\caption{Example of regression for a Wasserstein gradient flow of the energy $\mathcal{E}(\rho)=\iint W(x-y)\mathrm{d}\rho(x)\mathrm{d}\rho(y)+\int V(x) \mathrm{d}\rho(x)$, for $W(x)=\frac{1}{4}(|x|^2-1)^2$ and $V(x)=x_1^4+x_2^4+x_1^2x_2^2$. From left to right: true trajectories obtained from a discrete initial condition; reconstructed trajectories using $16$ equally spaced in time snapshots of the flow; true and predicted solutions at half time, using again $16$ snapshots; distance between true and predicted solution with increasing number of snapshots, measured with the Sinkhorn divergence $\mr{S}_\varepsilon$ at the initial, half and final time.} \label{fig:regression}
\end{figure}

\section*{Acknowledgements} 
 AN acknowledges funding by the Agence Nationale de la Recherche (ANR), project ANR-25-CE40-3242-01.  GT was supported by the European Union via the ERC AdG 101054420 EYAWKAJKOS project.

\appendix

\section{Debiased algorithms}

\subsection{Debiased quantization and classical barycenters}\label{sec:debbary}
We seek to compute particle approximation of the debiased barycenter problem:
\begin{equation}\label{eq:debbary}
\inf_{Y} \, \frac{\beta}{2} \sum_k \lambda^k {\mr{S}}_{\varepsilon}\left( \sum_{j} b_j \delta_{y_j}, \nu^k \right) = \inf_Y \left\{ \frac{\|Y\|^2_b}{2} - F(Y) \right\} \end{equation}
where 
\begin{equation}\label{eq:Fbeta}F(Y) = \frac{\|Y\|^2_b}{2}-\frac{\beta}{2} \sum_k \lambda^k {\mr{S}}_{\varepsilon}\left( \sum_{j} b_j \delta_{y_j}, \nu^k \right) \,.
\end{equation}
Here $\mr{S}_\varepsilon$ is the Sinkhorn divergence and $\beta$ is chosen so that $F(Y)$ is convex. By Toland's duality, this is equivalent to 

$$
\inf_Y \left\{ F^*(Y) - \frac{\|Y\|^2_b}{2}\right\}\,.
$$

The implicit Euler scheme with time step $\tilde{\tau} <1$ can be formulated as follows: given $Y^n$, $Y^{n+1}$ solves
\[
\inf_{Y} \left\{ \frac{\|Y - Y^n\|^2_b}{2\tilde{\tau}} - \frac{\|Y\|^2_b}{2}  + F^*(Y) \right\}\,.
\]
By convex duality, this is equivalent (up to a constant) to
\[
 \sup_Z \left\{ -\frac{\|Z - Y^n\|^2_b}{2(1-\tilde{\tau})} + \frac{\beta}{2} \sum_k \lambda^k \mr{S}_{\varepsilon}\left( \sum_{j} b_j \delta_{z_j}, \nu^k \right)  \right\}\,.
\]
Since $F$ defined in \eqref{eq:Fbeta} is convex, this is a strongly concave maximization problem for all $\tilde{\tau}<1$. Moreover, note that this problem is again a $K$-to-1 extrapolation with $t= (1-(1-\tilde{\tau})\beta)^{-1}$, but debiased. The algorithm has exactly the same structure as in the case discussed in Section \ref{sec:classical}.
In particular, the solutions $Y^{n+1}$ and $Z^{n+1}$ of the two problems are still related by
\[
Y^{n+1} = Y^{n} + \frac{\tilde{\tau}}{1-\tilde{\tau}} (Y^n - Z^{n+1})\,.
\]

\subsection{Debiased signed barycenters}\label{sec:debsignedbary}
We consider now the regularized signed barycenter problem where all Wasserstein distances are replaced by Sinkhorn: divergences
$$
\inf_{\mu} \, \left\{ t \sum_{l=1}^L \lambda_1^l \frac{ \mr{S}_\varepsilon\left(\mu, \nu_1^l \right)}{2} - (t-1)\sum_{k=1}^K\lambda_0^k \frac{\mr{S}_\varepsilon \left( \mu, \nu_0^k \right)}{2}\right\}
$$
Note that this can be seen as a direct generalization of debiased Sinkhorn barycenters \cite{janati2020debiased} to the signed setting.

As before we look for $\mu$ in the form $\mu = \sum_{j=1}^M b_j \delta_{z_j}$. In this case, we can still formulate the problem as a difference of convex functions. This time we write for some $\theta>0$
\[
\inf_{Z} \left\{ \tilde{F}_1(Z) - \tilde{F}_0(Z)\right\}
\]
where
\[
\tilde{F}_0(Z) =  \frac{\|Z\|_b^2}{2}  - \frac{t-1}{t+\theta} \sum_{k=1}^K \lambda_0^k \mr{S}_\varepsilon\left( \sum_{j=1}^M b_j \delta_{z_j}, \nu_0^k \right)\,, \]
\[
\tilde{F}_1(Z) = \frac{\|Z\|_b^2}{2} - \frac{t}{t+\theta}\sum_{l=1}^L \lambda_1^l \mr{S}_\varepsilon \left( \sum_{j=1}^M b_j \delta_{z_j}, \nu_1^l \right)\,.
\]

Let $\alpha$ be the modulus of semi-convexity of the self-transport term as in Lemma \ref{lem:convdebias}. Note that $\tilde{F}_1$ is convex as long as
\[
1 - \frac{t}{t+\theta}\left(1+\frac\alpha2\right) \geq 0 \implies \theta \geq t\frac{\alpha}{2}\,.
\]
This also guarantees that $\tilde{F}_0$ is strongly convex, since for $\theta \geq t\alpha/2$ we have that  
$$
1 - \frac{t-1}{t+\theta}\left(1+\frac\alpha2\right) \geq \frac{1}{t+\theta} \left(1+\frac\alpha2\right) > 0\,.
$$

To implement the entropic version of \eqref{eq:dc}, we only need to compute the (sub)gradients of $\tilde{F}^*_0$ and $\tilde{F}_1$.
Let us start by observing that 
$$
s = \frac{t+\theta}{1+\theta} \implies \frac{s-1}{s} = \frac{t-1}{t+\theta}\,.
$$
Therefore, 
\begin{equation}\label{eq:f0y}
    \tilde{F}^*_0(Y) = \frac{\|Y\|^2_b}{2} + \sup_Z \left\{-\frac{\|Z -Y\|^2_b}{2} +\frac{s-1}{2s} \sum_{q} \lambda_0^k \mr{S}_\varepsilon \left( \sum_j b_j \delta_{z_j} , \nu_0^q\right) \right\}
\end{equation}
Hence $\nabla \tilde{F}^*_0(Y)= Z^*$ is the unique solution to \eqref{eq:f0y}. Note that for $\theta \geq t \alpha/2$, we automatically have $s \leq t^*$, as in equation \eqref{eq:tstar}.

As for the subdifferential of $\tilde{F}_1$, we have
$$
\partial \tilde{F}_1(Z) = Z  - \frac{t}{t+\theta} \left( -\sum_{l=1}^L \lambda_1^l B^l(Z) + B^{\mathrm{self}}(Z)\right) $$
where $B^l(Z)$ is defined as in equation \eqref{eq:partialF1} replacing OT plans with entropic OT plans, and  $B^{\mathrm{self}}(Z)$ corresponds to the self-transport term.

\section{Proof of Proposition \ref{prop:dcconv}} \label{app:conv}

Using Fenchel duality and the strong convexity of $\tilde{F}_0$, we have
\[
\begin{aligned}
\tilde{F}_0(Z^n) -F_1(Z^n) &=  \langle Z^n,Y^n\rangle_b -\tilde{F}^*_0(Y^n) - F_1(Z^n) \\&\leq \langle Z^n,Y^n\rangle_b -\tilde{F}^*_0(Y^n) - F_1(Z^{n-1}) - \langle Y^n, Z^n - Z^{n-1}\rangle_b \\ &=  -\tilde{F}_0^*(Y^n) - F_1(Z^{n-1}) + \langle Y^n, Z^{n-1}\rangle_b \\ 
& \leq \tilde{F}_0(Z^{n-1})- F_1(Z^{n-1}) - \frac{t}{2}\|Z^n-Z^{n-1}\|^2_b\,.
\end{aligned}
\]
From this and the lower bound in the proof of Lemma \ref{lem:KtoLexistence} we deduces that the iterates $Z^n$ are bounded and $Z^n -Z^{n+1} \rightarrow 0$. By \eqref{eq:partialF1}, $(Y^n)_n$ is also bounded. For any convergent subsquence $(Z^{n_k})_k$, up to the extraction of a further subsequence we can assume that $(Y^{n_k+1})_k$ converges as well. The limit $Y^*$ and $Z^*$ satisfy
\[
Y^* \in \partial F_1(Z^*) \,, \quad Z^* \in \partial \tilde{F}_0^*(Y^*)\,,
\]
which are the first-order optimality conditions of the problem. In particular, $\mu = \sum_j b_j \delta_{z_j^*}$ satisfies the necessary optimality conditions in Proposition \ref{prop:crit}.

\section{Counterexample to uniqueness}\label{app:counter}

Let us consider the following setting
\[
\frac{\ed\nu_0}{\ed x} \sim \sum_{i=1}^4 \mc{X}_{B(p_i,1)}, \quad
\frac{\ed\nu_1^l}{\ed x} \sim \sum_{i=1}^2 \mc{X}_{B(q_i^l,\varepsilon)}
\]
where $p_1 =(1,1)$, $p_2=(1,-1)$, $p_3 = (-1,-1)$, $p_4=(-1,1)$, 
$q_1^1 = (d,0)$, $q_2^1=(-d,0)$, $q_1^2 = (0,d)$, $q_2^2 =(0,-d)$, 
with $d>1$ and $\varepsilon\in(0,1/2)$; see Figure \ref{fig:nonunique} for an illustration.

Suppose that $\mu$ is the unique solution of the problem. By symmetry, this needs 
to be invariant by rotations of $\pi/2$ around the origin. Furthermore, by 
Proposition \ref{prop:crit}, $\mu$ is also the solution of problem 
\eqref{eq:Kto1} with $\nu_1=\overline{\nu}_1$ being a linearized barycenter of 
$\nu_1^1$ and $\nu_1^2$ with base $\mu$. This means that the support of 
$\overline{\nu}_1$ is contained in the Minkowski average of the supports of 
$\nu_1^1$ and $\nu_1^2$. Again by symmetry, we deduce that 
$\overline{\nu}_1$ is necessarily of the form
\[
\overline{\nu}_1 = \frac{1}{4} \left( \overline{\nu} 
+ R^{\pi/2}_\# \overline{\nu}
+ R^{\pi}_\# \overline{\nu}
+ R^{3\pi/2}_\# \overline{\nu} \right)
\]
with
\[
\mr{supp}(\overline{\nu}) \subseteq 
\overline{B\left( \frac{q_1^1+q_2^2}{2},\varepsilon\right)}.
\]
From this and the properties of the metric extrapolation problem 
\cite{gallouet2025metric}, for $t$ sufficiently large and $\varepsilon$ sufficiently small, one can deduce that the 
extrapolation from $\nu_0$ to $\overline{\nu}_1$ must take the form
\[
\mu = \frac{1}{4}(\delta_{(a,a)} + \delta_{(a,-a)}
+\delta_{(-a,-a)}+\delta_{(-a,a)}) .
\]
We compute
\[
W_2^2(\mu,\nu_0)=2a^2-4a+\frac52,
\]
and
\[
W_2^2(\mu,\nu_1^1)
=
W_2^2(\mu,\nu_1^2)
=
2a^2
-2a\left(d+\frac{4\varepsilon}{3\pi}\right)
+d^2+\frac{\varepsilon^2}{2}.
\]
Therefore after minimization in $a$,
\begin{multline*}
t(W_2^2(\mu,\nu_1^1)+W_2^2(\mu,\nu_1^2))
-(t-1)W_2^2(\mu,\nu_0)
\\
=
2td^2+t\varepsilon^2-\frac52(t-1)
-\frac{
\left[
-4t\left(d+\frac{4\varepsilon}{3\pi}\right)
+4(t-1)
\right]^2
}{8(t+1)} .
\end{multline*}
To show that $\mu$ is not a minimizer, and hence that the minimizers of the 
signed barycenter problem may not be unique, we consider the competitor
\[
\overline{\mu}=\frac12(\delta_{(a,a)}+\delta_{(-a,-a)}).
\]
We compute
\[
W_2^2(\overline{\mu},\nu_0)
=
2a^2-2a-\frac{4a\sqrt2}{3\pi}+\frac52,
\]
and
\[
W_2^2(\overline{\mu},\nu_1^1)
=
W_2^2(\overline{\mu},\nu_1^2)
=
2a^2-2ad+d^2+\frac{\varepsilon^2}{2}.
\]
Therefore after minimizing over $a$,
\begin{multline*}
t(W_2^2(\overline{\mu},\nu_1^1)
+W_2^2(\overline{\mu},\nu_1^2))
-(t-1)W_2^2(\overline{\mu},\nu_0)
\\
=
2td^2+t\varepsilon^2-\frac52(t-1)
-\frac{
\left[
-4td+
2(t-1)\left(1+\frac{2\sqrt2}{3\pi}\right)
\right]^2
}{8(t+1)} .
\end{multline*}

The energy associated with $\mu$ is larger than that associated with 
$\overline{\mu}$ if and only if
\[
\frac{
\left[
-4t\left(d+\frac{4\varepsilon}{3\pi}\right)
+4(t-1)
\right]^2
}{8(t+1)}
\leq
\frac{
\left[
-4td+
2(t-1)\left(1+\frac{2\sqrt2}{3\pi}\right)
\right]^2
}{8(t+1)}.
\]

Since $d>1$ and supposing $t>1$ sufficiently large, this is equivalent to
\[
4t\left(d+\frac{4\varepsilon}{3\pi}\right)-4(t-1)
\leq
4td
-
2(t-1)\left(1+\frac{2\sqrt2}{3\pi}\right),
\]
or equivalently
\[
t\geq
\frac{3\pi-2\sqrt2}
{3\pi-2\sqrt2-8\varepsilon}.
\]

Therefore, for $t$ and $d$ sufficiently large and $\varepsilon$ sufficiently small, signed barycenters are not invariant by 
rotation by $\pi/2$ and hence they cannot be unique.
Note that this also shows that the minimizers do not form a convex set. In fact, if this was the case, we could always construct from any minimizer $\mu$ a rotation-invariant one by taking $(\mu + R^{\pi/2}_\# \mu+ R^{\pi}_\# \mu + R^{3\pi/2}_\# \mu)/4$.
\bibliographystyle{plain}
\bibliography{refs}   

\end{document}